\documentclass[amscd, leqno, 12pt]{amsart}

\usepackage{amssymb}
\usepackage{amsmath}
\usepackage{amscd}
\usepackage[all]{xy}
\usepackage{amsthm}
\usepackage{graphicx}
\usepackage{mathrsfs}
\usepackage[margin=30truemm]{geometry}
\usepackage{enumitem}
\usepackage{color}
\usepackage{comment}

\newtheorem{dfn}{Definition}[section]
\newtheorem{thm}[dfn]{Theorem}
\newtheorem{prop}[dfn]{Proposition}
\newtheorem{lem}[dfn]{Lemma}

\newtheorem*{cor*}{Corollary}

\theoremstyle{definition}

\newtheorem{exa}[dfn]{Example}
\newtheorem{conj}[dfn]{Conjecture}

\newcommand{\K}{K\"ahler }

\newcommand{\X}{\mathscr{X}}
\newcommand{\XX}{\mathscr{X}^{\iota}}
\newcommand{\xs}{\mathscr{X}/S}
\newcommand{\xss}{\mathscr{X}^{\iota}/S}
\newcommand{\Z}{\mathscr{Z}}
\newcommand{\zs}{\mathscr{Z}/S}

\newcommand{\mon}{\operatorname{Mon}^2}
\newcommand{\ktm}{\operatorname{KT}(M)}

\newcommand{\vol}{\operatorname{Vol}}

\newcommand{\rk}{\operatorname{rk}}
\newcommand{\sign}{\operatorname{sign}}

\newcommand{\G}{\mu_2}
\newcommand{\D}{\varDelta}
\newcommand{\cpq}[2]{c_1( R^{#2}h_* \Omega^{#1}_{\zs}, h_{L^2} )}

\begin{document} 

\title[Analytic torsion for IHS 4-folds with involution]{Analytic torsion for irreducible holomorphic symplectic fourfolds with involution, III: relation with the BCOV invariant}
\author{Dai Imaike}
\address{
Department of Mathematics,
Faculty of Science,
Kyoto University,
Kyoto 606-8502,
Japan}
\email{imaike.dai.22s@st.kyoto-u.ac.jp}

\maketitle

\numberwithin{equation}{section}

\begin{abstract}
	A Calabi-Yau 4-fold of Camere-Garbagnati-Mongardi is a crepant resolution of the quotient of a hyperk\"ahler 4-fold by an antisymplectic involution.
	In this paper, we compare two different types of holomorphic torsion invaiants;
	one is the BCOV invariant of the Calabi-Yau 4-fold of Camere-Garbagnati-Mongardi,
	and the other is the invariant of the corresponding $K3^{[2]}$-type manifold with involution introduced by the author in the preceding papers. 
	As an application, in some special cases, we show that the BCOV invariant of those Calabi-Yau 4-folds is expressed as the Petersson norm of a certain Borcherds product.
\end{abstract}

\setcounter{section}{-1}

\section{Introduction}\label{s-0}

	This paper is the third of a series of three papers investigating equivariant analytic torsion for manifolds of $K3^{[2]}$-type with antisymplectic involution.
	In \cite{I1}, we constructed an invariant of $K3^{[2]}$-type manifolds with involution by using equivariant analytic torsion
	and we proved that the invariant satisfies a certain curvature equation.
	In \cite{I2}, we proved the algebraicity of the singularity of the invariant
	and compared our invariant with the invariant of 2-elementary K3 surfaces constructed by Yoshikawa \cite{MR2047658}.
	In this paper, we compare our invariant with the BCOV invariant for a class of Calabi-Yau 4-folds introduced by Camere-Garbagnati-Mongardi \cite{MR3928256}.
	Let us recall BCOV torsion and BCOV invariants.
	
	In their study of mirror symmetry at genus one, Bershadsky-Cecotti-Ooguri-Vafa \cite{MR1301851} introduced the following combination of Ray-Singer analytic torsions:
	$$
		\mathcal{T}_{BCOV}(Z, \omega_Z) = \prod_{p \geqq 0} \tau(\overline{\Omega}^p_Z)^{(-1)^p p},
	$$
	where $\tau(\overline{\Omega}^p_Z)$ is the analytic torsion of $\Omega^p_Z$ on a compact Ricci-flat \K manifold $Z$. 
	This combination of analytic torsions is called the BCOV torsion.
	The conjecture of Bershadsky-Cecotti-Ooguri-Vafa claims the equivalence of the following two functions via the mirror map; one is the generating series of the genus one Gromov-Witten invariants of a Calabi-Yau 3-fold, and the other is the BCOV torsion viewed as a function on the base space of its mirror family of Calabi-Yau 3-folds.
	
	After Bershadsky-Cecotti-Ooguri-Vafa, in mathematics, the BCOV invariant for Calabi-Yau manifolds was introduced by Fang-Lu-Yoshikawa \cite{MR2454893} in dimension $3$, and by Eriksson-Freixas i Montplet-Mourougane \cite{MR4255041}, \cite{MR4475251} in arbitrary dimension.
	Besides the BCOV torsion, they introduced certain correction terms consisting of a Bott-Chern secondary class and various covolumes of cohomology lattices to get a genuine invariant (See \S \ref{ss-2-0}).
	Then the BCOV invariant of the mirror family of a Calabi-Yau hypersurface in projective space is determined by Fang-Lu-Yoshikawa in dimension $3$, and by Eriksson-Freixas i Montplet-Mourougane in arbitrary dimension.
	Moreover, Eriksson-Freixas i Montplet-Mourougane extend the mirror symmetry conjecture at genus one to arbitrary dimension \cite{MR4475251}.
	On the other hand, the genus one Gromov-Witten invariants of a Calabi-Yau hypersurface of a projective space was determined by Zinger \cite{MR2403808}, \cite{MR2505298}.
	By comparing the formula for the BCOV invariant and the generating function of the genus one Gromov-Witten invariants, the mirror symmetry conjecture at genus one between a Calabi-Yau hypersurface of $\mathbb{P}^n$ and its mirror family is established by Zinger \cite{MR2403808}, \cite{MR2505298} and Fang-Lu-Yoshikawa \cite{MR2454893} for $n=4$, and by Zinger \cite{MR2403808}, \cite{MR2505298} and Eriksson-Freixas i Montplet-Mourougane \cite{MR4255041}, \cite{MR4475251} for $n \geqq 5$.
	Based on the work of Popa \cite{MR3003261}, this result is extended to the case of a smooth complete intersection of two cubic hypersurfaces of $\mathbb{P}^5$ by Pochekai \cite{pochekai}, Eriksson-Pochekai \cite{eriksson-pochekai}.
	
	As for the BCOV invariant of complex tori of higher dimension and hyperk\"ahler manifolds, its triviality is proved by Eriksson-Freixas i Montplet-Mourougane \cite{MR4255041}.
	However, for Ricci-flat compact \K manifolds with torsion canonical bundle, the BCOV torsion is not necessarily trivial (\cite{MR2047658}). % when it is viewed as a function on the moduli space.
	%Yoshikawa \cite{MR2047658} proved that the analytic torsion (in fact the BCOV torsion) of Ricci-flat Enriques surfaces is expressed as the Petersson norm of the Borcherds $\Phi$-function \cite{MR1396773}.
	%On the other hand, very recently, Oberdieck \cite{oberdieck2023non} proves the explicit formula for the Gromov-Witten invariants of an Enriques surface conjectured by Klemm-Mari\~{n}o \cite{MR2391190} and he shows that the generating series of the genus one Gromov-Witten invariants of an Enriques surface is derived from a certain power of the Borcherds $\Phi$-function.
	For Enriques surfaces, one has the equivalence of the BCOV torsion and the genus one Gromov-Witten invariants by Oberdieck \cite{oberdieck2023non} and Yoshikawa \cite{MR2047658}.
	
	%In \cite{MR3003261}, Popa proved an explicit formula for the genus one Gromov-Witten invariants of smooth complete intersection in the projective space.
	%After Popa's formula, Eriksson-Pochekai \cite{eriksson-pochekai} established the mirror symmetry conjecture at genus one between a smooth complete intersection of two cubic hypersurfaces of $\mathbb{P}^5$ and its mirror family constructed by Pochekai \cite{pochekai}.
	
	In this paper, we study the BCOV invariant of the Calabi-Yau 4-folds constructed by Camere-Garbagnati-Mongardi \cite{MR3928256},
	who obtained their Calabi-Yau 4-folds as a crepant resolution of the quotient of a hyperk\"ahler 4-fold by an antisymplectic involution. 
	Under a certain assumption, we show that the BCOV invariant of the Calabi-Yau 4-folds of Camere-Garbagnati-Mongardi, viewed as a function on the moduli space, coincides with the invariant of the $K3^{[2]}$-type manifolds with involution constructed by the author \cite{I1}, up to a universal constant.
	In particular, combining this result, the result in \cite{I2} and Yoshikawa's result \cite{MR2047658}, we show that the BCOV invariant of the universal covering of the Hilbert scheme of 2-points on an Enriques surface is expressed as the Petersson norm of the Borcherds $\Phi$-function.
	Furthermore, we show that the BCOV invariant of the Calabi-Yau 4-fold of Camere-Garbagnati-Mongardi obtained from a 2-elementary K3 surface with no fixed curves of positive genus, is expressed as the Petersson norm of the Borcherds product of Yoshikawa \cite{MR2674883} and Gritsenko \cite{MR3859399}.
	Let us explain our results in more details.
	
\medskip

	Let $X$ be a manifold of $K3^{[2]}$-type.
	Namely, $X$ is an irreducible holomorphic symplectic manifold deformation equivariant to the Hilbert scheme of $2$ points of a K3 surface.
	A holomorphic involution $\iota : X \to X$ satisfying $\iota^* \eta =-\eta$ is called antisymplectic.
	In this paper, involutions on an irreducible holomorphic symplectic manifold are always assumed to be antisymplectic.

	Let us briefly recall the invariant $\tau(X, \iota)$ of a $K3^{[2]}$-type manifold with antisymplectic involution $(X, \iota)$ constructed in \cite{I1}.
	Let $h_X$ be an $\iota$-invariant \K metric on $X$.
	For simplicity, we assume that $h_X$ is Ricci-flat and the volume of $(X, h_X)$ is $1$.
	Let $\tau_{\iota}(\bar{\Omega}^1_X)$ be the equivariant analytic torsion of the cotangent bundle $\Omega^1_X$ with respect to $h_X$ and the $\G$-action induced from $\iota$.
	The fixed locus of $\iota : X \to X$ is denoted by $X^{\iota}$.
	By \cite{MR2805992}, $X^{\iota}$ is a possibly disconnected smooth complex surface.
	Let $\tau(\bar{\mathcal{O}}_{X^{\iota}})$ be the analytic torsion of the trivial line bundle $\mathcal{O}_{X^{\iota}}$ with respect to $h_X|_{X^{\iota}}$.
	The volume of $X^{\iota}$ with respect to $h_X|_{X^{\iota}}$ is denoted by $\vol(X^{\iota}, h_X|_{X^{\iota}})$.
	We denote by $\vol_{L^2}\left( H^1(X^{\iota}, \mathbb{Z}), h_X|_{X^{\iota}} \right)$ the covolume of $H^1(X^{\iota}, \mathbb{Z})$
	with respect to the $L^2$-metric on $H^1(X^{\iota}, \mathbb{R})$ induced from $h_X|_{X^{\iota}}$.
	We define a real number $\tau(X, \iota)$ by
	\begin{align*}
		\tau(X, \iota)=\tau_{\iota}(\bar{\Omega}_X^1) 
		 \tau(\bar{\mathcal{O}}_{X^{\iota}})^{-2} \vol(X^{\iota}, h_X|_{X^{\iota}})^{-2} \vol_{L^2}(H^1(X^{\iota}, \mathbb{Z}), h_X|_{X^{\iota}}).
	\end{align*}
	By \cite{I1}, $\tau(X, \iota)$ is independent of the choice of a normalized $\iota$-invariant Ricci-flat \K metric $h_X$.
	Hence $\tau(X, \iota)$ is an invariant of $(X, \iota)$ (See \ref{ss-1-3} for more details).

	Let $M_0$ be a primitive hyperbolic 2-elementary sublattice of the K3 lattice $L_{K3}$.
	Let $(Y, \sigma)$ be a 2-elementary K3 surface of type $M_0$.
	Namely, $(Y, \sigma)$ is the pair consisting of a K3 surface $Y$ and an antisymplectic involution $\sigma : Y \to Y$ such that there exists an isometry $\alpha : H^2(Y, \mathbb{Z}) \to L_{K3}$ with $\alpha(H^2(Y, \mathbb{Z})^{\sigma})=M_0$.
	The Hilbert scheme of 2-points of $Y$ is denoted by $X_Y$
	and $\iota_Y : X_Y \to X_Y$ is the involution induced from $\sigma$.
	Then $(X_Y, \iota_X)$ is a manifold of $K3^{[2]}$-type with antisymplectic involution.
	The blowup of $X_Y$ along the fixed locus is denoted by $\widetilde{X}_Y$
	and the involution of $\widetilde{X}_Y$ induced from $\iota_Y$ is denoted by $\tilde{\iota}_Y$.
	Camere-Garbagnati-Mongardi \cite{MR3928256} proved that its quotient $Z_Y := \widetilde{X}_Y / \tilde{\iota}_Y$ is a Calabi-Yau 4-fold
	and the canonical map $Z_Y \to X_Y/ \iota_Y$ is a crepant resolution.

	Let $\tau_{BCOV}(Z_Y)$ be the BCOV invariant constructed by Eriksson-Freixas i Montplet-Mourougane \cite{MR4255041}, \cite{MR4475251}.
	One of the main results of this paper is a comparison theorem between $\tau(X_Y, \iota_Y)$ and $\tau_{BCOV}(Z_Y)$.

	\begin{thm}[Theorem \ref{t-2-3-4}]\label{t-0-1}
		Suppose that $\rk(M_0) \leqq 17$ 
		and that $\Delta(M_0^{\perp})$, the set of roots of $M_0^{\perp}$, consists of a unique $O(M_0^{\perp})$-orbit.
		Then there exists a constant $C_{M_0} >0$ depending only on $M_0$ such that for any 2-elementary K3 surface $(Y, \sigma)$ of type $M_0$
		$$
			\tau_{BCOV}(Z_Y) =C_{M_0} \tau (X_Y, \iota_Y).
		$$
	\end{thm}
	
	Fu-Zhang \cite{MR4549963} introduced the BCOV invariant of Calabi-Yau varieties with canonical singularities
	and they prove its birational equivalence.
	After the comparison theorem between $\tau$ and $\tau_{BCOV}$ such as Theorem \ref{t-0-1} and the result of Fu-Zhang \cite{MR4549963},
	it is natural to ask if the birational invariance of $\tau$ holds (See Conjecture \ref{c4-3-1}).
	
	\medskip
	
	Now, let us explain some applications of Theorem \ref{t-0-1}.
	Assume that $M_0= \Lambda(2)$, where $\Lambda$ is an even unimodular lattice of signature $(1,9)$.
	%Assume that $M_0 = U(2) \oplus E_8(2)$,
	%where $U$ is the even unimodular lattice of signature $(1,1)$
	%and $E_8$ is the negative definite even unimodular lattice associated with the Dynkin diagram $E_8$.
	Here, for a lattice $L =(\mathbb{Z}^r, (\cdot, \cdot)_{L})$, we define $L(k) =(\mathbb{Z}^r, k(\cdot, \cdot)_{L})$. 
	For a 2-elementary K3 surface $(Y, \sigma)$, its fixed locus $Y^{\sigma}$ is empty and the quotient $Y/ \sigma$ is an Enriques surface.
	By \cite{MR3928256},  \cite{MR2863907}, $Z_Y$ is the universal cover of the Hilbert scheme of 2-points on the Enriques surface $Y/ \sigma$.
	
	In \cite{MR1396773}, \cite{MR1625724}, Borcherds introduced an automorphic form $\Phi$ of weight $4$ on the period domain for Enriques surfaces
	characterizing the discriminant locus.
	We denote by $\| \Phi ([S]) \|$ the Petersson norm of the Borcherds $\Phi$ function evaluated at the period of an Enriques surface $S$.

	\begin{thm}[Theorem \ref{t2-2-3-2}]\label{t-0-2}
		There exists a constant $C>0$ such that for any Enriques surface $S$
		\begin{align*}
			\tau_{BCOV}( \widetilde{S^{[2]}} ) = C \| \Phi ( [S] ) \|,
		\end{align*}
		where $\widetilde{S^{[2]}}$ is the universal cover of the Hilbert scheme of 2-points on $S$.
	\end{thm}

	Let $k=1, \dots, 6$
	 and assume that $M_0=\Lambda_k(2)^{\perp}$, 
	where $\Lambda_k$ is an odd unimodular lattice of signature $(2, 10-k)$.
	%and the orthogonal complement is considered in $L_{K3}$. 
	Then the fixed locus $Y^{\sigma}$ of a 2-elementary K3 surface of type $M_0=\Lambda_k(2)^{\perp}$ is the disjoint union of $k$-rational curves. 
	
	A reflective modular form $\Phi_k$ of weight $k+4$ on the period domain for 2-elementary K3 surfaces of type $M_0=\Lambda_k(2)^{\perp}$ was constructed independently by Yoshikawa \cite{MR2674883} and Gritsenko \cite{MR3859399}.
	By \cite{MR4476107}, $\Phi_k$ characterizes the discriminant locus on the moduli space of certain log-Enriques surfaces.
	We denote by $\| \Phi_k ([Y]) \|$ the Petersson norm of $\Phi_k$ evaluated at the period of a 2-elementary K3 surface $(Y, \sigma)$ of type $M_0$.

	\begin{thm}[Theorem \ref{t2-2-3-4}]\label{t-0-3}
		There exists a constant $C_k$ depending only on $1 \leqq k \leqq 6$ such that for any 2-elementary K3 surface $(Y, \sigma)$ of type $M_0=\Lambda_k(2)^{\perp}$
		\begin{align*}
			\tau_{BCOV}(Z_Y) = C_k \| \Phi_k ([Y]) \|^{k+1}.
		\end{align*}
	\end{thm}

	After the work of Eriksson-Freixas i Montplet-Mourougane \cite{MR4475251},
	to our knowledge, Theorems \ref{t-0-2} and \ref{t-0-3} are the first to give an explicit formula for the BCOV invariant of Calabi-Yau 4-folds,
	whose moduli space has dimension strictly greater than $1$.

\medskip

	This paper is organized as follows. 
	In Section 1, we briefly recall the invariant $\tau$ of $K3^{[2]}$-type manifolds with antisymplectic involution constructed in \cite{I1} and \cite{I2}.
	In Section 2, we show that the invariant $\tau$ of $K3^{[2]}$-type manifolds with antisymplectic involution satisfies the same curvature equation as the one for the BCOV invariant of the Calabi-Yau 4-folds of Camere-Garbagnati-Mongardi, and we prove Theorem \ref{t-0-1}.
	In Section 3, we prove Theorems \ref{t-0-2} and \ref{t-0-3}.

\medskip
	
	Acknowledgements.
	 I would like to express my gratitude to my advisor, Professor Ken-Ichi Yoshikawa, for suggesting this problem and for his help and encouragement. 
	 This work was supported by JSPS KAKENHI Grant Number 23KJ1249.

\section{Preliminary}\label{s-1}

	In this section, we recall the construction and some properties of the invariant of $K3^{[2]}$-type manifolds with antisymplectic involution discussed in \cite{I1} and \cite{I2}.

\subsection{Equivariant analytic torsion}

	We recall the definition of equivariant analytic torsion for compact \K manifolds with holomorphic involution.
	
	Let $M$ be a compact complex manifold of dimension $m$,
	and let $i : M \to M$ be a holomorphic involution of $M$.
	Let $\G$ be the group generated by the order 2 element $i$.
	In what follows, we consider the $\mu_2$-action on $M$ induced by $i$.
	Let $h_M$ be an $i$-invariant \K metric on $M$.
	Its associated \K form is denoted by $\omega_M$.
	The space of smooth $(p,q)$-forms on $M$ is denoted by $A^{p,q}(M)$.
	
	Let $E$ be a $\G$-equivariant holomorphic vector bundle on $M$,
	and $h_E$ a $\G$-invariant hermitian metric on $E$.
	The space of $E$-valued smooth $(p,q)$-forms on $M$ is denoted by $A^{p,q}(M, E)$ or $A^{p,q}(E)$.
	The metrics $h_M$ and $h_E$ induce a $\G$-invariant hermitian metric $h$ on the complex vector bundle $\wedge^{p,q}T^*M \otimes E$.
	The $L^2$-metric on $A^{p,q}(M, E)$ is defined by
	$$
		\langle \alpha, \beta \rangle_{L^2} = \int_M h(\alpha, \beta) \frac{\omega_M^m}{m !}, \quad \alpha, \beta \in A^{p,q}(M, E).
	$$
	The Dolbeault operator of $E$ is denoted by $\bar{\partial}_E : A^{p,q}(M, E) \to A^{p,q+1}(M, E)$,
	and its formal adjoint is denoted by $\bar{\partial}_E^* : A^{p,q}(M, E) \to A^{p,q-1}(M, E)$.
	We define the Laplacian $\D_{p,q}$ acting on $A^{p,q}(M, E)$ by 
	$
		\D_{p,q} = (\bar{\partial}_E + \bar{\partial}_E^*)^2 : A^{p,q}(M, E) \to A^{p,q}(M, E).
	$
	We denote the spectrum of $\D_{p,q}$ by $\sigma(\D_{p,q})$, and the eigenspace of $\D_{p,q}$ associated with an eigenvalue $\lambda \in \sigma(\D_{p,q})$ by $E_{p,q}(\lambda)$.
	Note that $\sigma(\D_{p,q})$ is a discrete subset contained in $\mathbb{R}_{\geqq 0}$.
	Moreover $E_{p,q}(\lambda)$ is finite dimensional.
	Let $g \in \G$. 
	The spectral zeta function is defined by 
	$$
		\zeta_{p,q,g}(s) = \sum_{\lambda \in \sigma(\D_{p,q}) \setminus \{ 0 \}} \lambda^{-s} \operatorname{Tr} (g|_{E_{p,q}(\lambda)}) \quad (s \in \mathbb{C}, \operatorname{Re} s >n).
	$$
	Note that $\zeta_{p,q,g}(s)$ converges absolutely on the domain $\operatorname{Re} s > n$
	and extends to a meromorphic function on $\mathbb{C}$ which is holomorphic at $s=0$.
	
	\begin{dfn}\label{d-1-2}
		Let $g \in \G$.
		The equivariant analytic torsion of $\overline{E} = (E, h_E)$ on $(M, \omega_M)$ is defined by
		$$
			\tau_g(\overline{E}) = \exp \bigl\{ - \sum_{q=0}^n (-1)^q q \zeta'_{0,q,g}(0) \bigr\} .
		$$
	\end{dfn}
	
	If $g=1$, it is the analytic torsion of $\overline{E}$ and is denoted by $\tau(\overline{E})$ instead of $\tau_1(\overline{E})$.
	
	We denote by $H^q(M, E)_{\pm}$ the $(\pm 1)$-eigenspace of $i^* : H^q(M, E) \to H^q(M, E)$.
	We set
	$
		\lambda_{\pm}(E) = \bigotimes_{q \geqq 0} ( \det H^q(M, E)_{\pm} )^{(-1)^q}. 
	$ 
	We define the equivariant determinant of the cohomologies of $E$ by
	$$
		\lambda_{\G}(E) = \lambda_{+}(E) \oplus \lambda_{-}(E). 
	$$
	
	By Hodge theory, we can identify $H^q(M, E)$ with the space of $E$-valued harmonic $(0, q)$-forms on $M$.
	Under this identification, the cohomology $H^q(M, E)$ is endowed with the $\G$-invariant hermitian metric induced from the $L^2$-metric on $A^{0,q}(M, E)$.
	It induces the hermitian metric $\| \cdot \|_{\lambda_{\pm}(E), L^2}$ on $\lambda_{\pm}(E)$.
	We define the equivariant metric on $\lambda_{\G}(E)$ by
	$$
		\| \alpha \|_{\lambda_{\G}(E), L^2}(i) = \| \alpha_+ \|_{\lambda_{\pm}(E), L^2} \cdot \| \alpha_- \|_{\lambda_{\pm}(E), L^2}^{-1} \quad ( \alpha=(\alpha_+, \alpha_-) \in \lambda_{\G}(E), \alpha_+, \alpha_- \neq 0 ),
	$$
	and call it the equivariant $L^2$-metric on $\lambda_{\G}(E)$.
	We define the equivariant Quillen metric on $\lambda_{\G}(E)$ by
	$$
		\| \alpha \|^2_{\lambda_{\G}(E), Q}(i) = \tau_g(\overline{E}) \| \alpha \|^2_{\lambda_{\G}(E), L^2}(i) .
	$$

\subsection{Manifolds of $K3^{[2]}$-type and antisymplectic involutions}
	
	A simply-connected compact \K manifold $X$ is an irreducible holomorphic symplectic manifold if there exists an everywhere non-degenerate holomorphic 2-form $\eta$ such that $H^0(X, \Omega_X^2)$ is generated by $\eta$.
	By \cite[Proposition 23.14 and Remark 23.15]{MR1963559}, there exist a unique primitive integral quadric form $q_X$ on $H^2(X, \mathbb{Z})$ of signature $(3, b_2(X)-3)$ 
	and a positive rational number $c_X \in \mathbb{Q}_{\geqq 0}$ such that $q_X(\alpha)^n = c_X \int_X \alpha^{2n}$ for any $\alpha \in H^2(X, \mathbb{Z})$. 
	If $b_2(X)=6$, we also require that $q_X(\omega)>0$ for any \K class $\omega$. 
	The quadric form $q_X$ is called the Beauville-Bogomolov-Fujiki form.
	Let $( \cdot , \cdot ) : H^2(X, \mathbb{Z}) \times H^2(X, \mathbb{Z}) \to \mathbb{Z}$ be the integral bilinear form corresponding to $q_X$.
	A holomorphic involution $\iota : X \to X$ satisfying $\iota^* \eta =-\eta$ is called antisymplectic.
	In this paper, involutions on an irreducible holomorphic symplectic manifold are always assumed to be antisymplectic.
	An irreducible holomorphic symplectic manifold $X$ is a manifold of $K3^{[2]}$-type if $X$ is deformation equivariant to the Hilbert scheme of length $2$ zero-dimensional subschemes of a K3 surface.
	We denote by $U$ an even unimodular lattice of signature $(1,1)$
	and denote by $E_8$ the negative definite even unimodular lattice associated with the Dynkin diagram $E_8$.
	Then $H^2( X, \mathbb{Z} )$ is isomorphic to 
	$$
		L_2 := L_{K3} \oplus \mathbb{Z} e,
	$$
	where $L_{K3} = E_8^{\oplus 2} \oplus U^{\oplus 3}$ is the $K3$ lattice and $e$ is an element with $e^2 =-2$ and $(e, L_{K3})=0$.
	
	Let $X_1, X_2$ be irreducible holomorphic symplectic manifolds.
	Recall that a parallel-transport operator $f : H^2(X_1, \mathbb{Z}) \to H^2(X_2, \mathbb{Z})$ is an isomorphism such that
	there exist a family $p : \mathcal{X} \to B$ of irreducible holomorphic symplectic manifolds over a possibly reducible analytic base $B$, 
	two points $b_1, b_2 \in B$, and a continuous path $\gamma : [0, 1] \to B$ with $\gamma(0)=b_1, \gamma(1)=b_2$ such that $p^{-1}(b_i) \cong X_i$ $(i=1,2)$ and that the parallel-transport in the local system $R^2p_*\mathbb{Z}$ induces $f : H^2(X_1, \mathbb{Z}) \to H^2(X_2, \mathbb{Z})$.

	Let $X$ be an irreducible holomorphic symplectic manifold and let $\Lambda$ be a lattice isomorphic to $H^2( X, \mathbb{Z} )$.
	A parallel-transport operator $g : H^2(X, \mathbb{Z}) \to H^2(X, \mathbb{Z})$ is called a monodromy operator.
	The subgroup $\mon(X)$ of the isometry group $O(H^2(X, \mathbb{Z}))$ consisting of all monodromy operators of $X$ is called the monodromy group.
	For an isometry $\alpha : H^2( X, \mathbb{Z} ) \to \Lambda$, we set $\mon( \Lambda ) = \alpha \circ \mon(X) \circ \alpha^{-1}$.
	If $X$ is a K3 surface or a manifold of $K3^{[2]}$-type,
	it follows from \cite[Theorem 9.1]{MR2964480} that the group $\mon(\Lambda)$ is a normal subgroup of $O(\Lambda)$ and is independent of the choice of $(X,\alpha)$.
	
	Let $O^+(\Lambda)$ be the subgroup of $O(\Lambda)$ consisting of the isometries of real spinor norm $+1$.
	If $X$ is a K3 surface or a manifold of $K3^{[2]}$-type,
	it follows from \cite[Lemma 9.2]{MR2964480} and \cite[Theorem A]{MR0849050} that $O^+(\Lambda) = \mon(\Lambda)$.
	We set $\tilde{\mathscr{C}}_{\Lambda}=\{ x \in \Lambda_{\mathbb{R}} ; x^2>0 \}$.
	Since $\operatorname{sign}(\Lambda) = (3, \operatorname{rank} \Lambda -3)$, we have $H^2( \tilde{\mathscr{C}}_{\Lambda}, \mathbb{Z} ) \cong \mathbb{Z}$ by \cite[Lemma 4.1]{MR2964480}.
	A generator of $H^2( \tilde{\mathscr{C}}_{\Lambda}, \mathbb{Z} )$ is called an orientation class.
	By \cite[\S 4]{MR2964480}, any isometry of real spinor norm $+1$ preserves an orientation class. 
	
	A pair $(X, \alpha)$ is called a marked manifold of $K3^{[2]}$-type if $X$ is a manifold of $K3^{[2]}$-type and $\alpha : H^2(X, \mathbb{Z}) \to L_2$ is an isometry.
	Let $\mathfrak{M}_{L_2}$ be the moduli space of marked manifolds of $K3^{[2]}$-type constructed in \cite[Definition 25.4]{MR1963559}.
	We fix a connected component $\mathfrak{M}_{L_2}^{\circ}$ of $\mathfrak{M}_{L_2}$.
	
	For a hyperbolic sublattice $M$ of $L_2$ and an involution $\iota_M \in \mon(L_2)$,
	the pair $(M, \iota_M)$ is an admissible sublattice of $L_2$ if the invariant sublattice $(L_2)^{\iota_M}$ of $\iota_M$ is equal to $M$.
	By \cite[Figures 1 and 2]{MR3542771}, there are $150$ types of admissible sublattice of $L_2$.
	Let $(X, \iota)$ be a manifold of $K3^{[2]}$-type with involution.
	Then $\iota$ is of type $M$ if there exists an isometry $\alpha : H^2(X, \mathbb{Z}) \to L_2$ such that $(X,\alpha) \in \mathfrak{M}_{L_2}^{\circ}$ and $\alpha \circ \iota^* = \iota_M \circ \alpha$.  
	We set
	$$
		\Delta(M) = \left\{ \delta \in M ; \delta^2 =-2, \text{or } \delta^2=-10,  (\delta, L_2)=2\mathbb{Z} \right\} \quad \text{and} \quad \tilde{\mathscr{C}}_M = \{ x \in M_{\mathbb{R}} ; x^2>0 \}.
	$$ 
	For $\delta \in \Delta(M)$, we define $\delta^{\perp}=\{ x \in \tilde{\mathscr{C}}_M : (x, \delta)=0 \}$.
	Then $\bigcup_{\delta \in \Delta(M)} \delta^{\perp}$ is a closed subset of $\tilde{\mathscr{C}}_M$
	and a connected component of $\tilde{\mathscr{C}}_M \setminus \bigcup_{\delta \in \Delta(M)} \delta^{\perp}$ is called a K\"ahler-type chamber of $M$.
	The set of all K\"ahler-type chambers of $M$ is denoted by $\ktm$. 
	We set
	$$
		\Gamma(M) = \{ \sigma \in \mon(L_2) ; \iota_M \circ \sigma = \sigma \circ \iota_M \} \text{ and } \Gamma_M = \{ \sigma|_{M} \in O(M) ; \sigma \in \Gamma(M) \}.
	$$
	Then $\Gamma_M$ acts on $\ktm$.
	%By \cite[Theorem 9.11]{MR3519981}, there is a bijection from $\ktm / \Gamma_M$ to the set of all deformation types of $K3^{[2]}$-type manifolds with antisymplectic involution of type $M$.
	
	Let $(M, \iota_M)$ be an admissible sublattice of $L_2$ and let $\mathcal{K} \in \ktm$.
	Let $(X, \iota)$ be a manifold of $K3^{[2]}$-type with involution. 
	A marking $\alpha : H^2(X, \mathbb{Z}) \to L_2$ of $(X, \iota)$ is said to be admissible for $(M, \mathcal{K})$ if $(X,\alpha) \in \mathfrak{M}_{L_2}^{\circ}$, $\alpha \circ \iota^* = \iota_M \circ \alpha$, and $\alpha(\mathcal{K}_X^{\iota}) \subset \mathcal{K}$,
	where $\mathcal{K}_X^{\iota} \subset H^2(X, \mathbb{R})^{\iota}$ is the set of $\iota$-invariant \K classes of $X$. 
	Moreover $\iota$ is called an involution of type $(M, \mathcal{K})$ if there exists an admissible marking for $(M, \mathcal{K})$.
	We denote by $\tilde{\mathcal{M}}_{M, \mathcal{K}}$ the set of all isomorphism classes of $K3^{[2]}$-type manifolds $(X, \iota)$ with involution of type $(M, \mathcal{K})$.
	By \cite[Theorem 9.11]{MR3519981}, any two isomorphism classes of $\tilde{\mathcal{M}}_{M, \mathcal{K}}$ are deformation equivalent.
	
	Let $L$ be a lattice of signature $(2, n)$.
	Then
	$$
		\Omega_{L} = \{ [\eta] \in \mathbb{P}(L_{\mathbb{C}}) ; (\eta, \eta)=0, (\eta, \bar{\eta}) >0 \}.
	$$
	consists of two connected components, both of which are isomorphic to a bounded symmetric domain of type IV of dimension $n$.
	
	Fix $h \in \mathcal{K} \cap M$.
	Then $h^{\perp_{L_2}}$ is a lattice of signature $(2, 20)$ containing $M^{\perp_{L_2}}$,
	where the orthogonal complement of $h$ and $M$ are considered in $L_2$. %, and $\Omega_{h^{\perp}}$ consists of two connected components.
	By \cite[(4.1)]{MR2964480}, the connected component $\mathfrak{M}_{L_2}^{\circ}$ determines the connected component $\Omega_{h^{\perp}}^+$ of $\Omega_{h^{\perp}}$.
	%$$
	%	\Omega_{h^{\perp}} = \{ [\eta] \in \mathbb{P}( (h^{\perp})_{\mathbb{C}}) ; (\eta, \eta)=0, (\eta, \bar{\eta}) >0 \}.
	%$$
	Namely, for any $(X, \alpha) \in \mathfrak{M}_{L_2}^{\circ}$ and $p = [\sigma] \in \Omega_{h^{\perp}}^+$, the orientation of $\tilde{\mathscr{C}}_{X} = \{ x \in H^2(X, \mathbb{R}) ; q_X(x)>0 \}$ determined by the real and the imaginary part of  the holomorphic 2-form on $X$ and a \K class of $X$ is compatible with 
	the orientation of $\tilde{\mathscr{C}_{L_2}} = \{ x \in L_{2, \mathbb{R}} ; x^2>0 \}$ determined by the real 3-dimensional vector space $W := \operatorname{Span}_{\mathbb{R}} \{ \operatorname{Re} \sigma, \operatorname{Im} \sigma, h \}$ associated to the basis $\{ \operatorname{Re} \sigma, \operatorname{Im} \sigma, h \}$ via the isomorphism $\tilde{\mathscr{C}}_X \cong \tilde{\mathscr{C}_{L_2}}$ induced from the marking $\alpha$.
	
	%Set
	%$$
	%	\Omega_{M^{\perp}} = \{ [\eta] \in \mathbb{P}(M^{\perp}_{\mathbb{C}}) ; (\eta, \eta)=0, (\eta, \bar{\eta}) >0 \}.
	%$$
	Let $\Omega^+_{M^{\perp}}$ be the connected component of $\Omega_{M^{\perp}}$
	satisfying $\Omega^+_{M^{\perp}} \subset \Omega^+_{h^{\perp}}$.
	Let $(X, \iota) \in \tilde{\mathcal{M}}_{M, \mathcal{K}}$ and choose an admissible marking $\alpha : H^2(X, \mathbb{Z}) \to L_2$ for $(M, \mathcal{K})$.
	Since $\alpha( H^{2,0}(X) ) \subset M^{\perp}_{\mathbb{R}} \subset (h^{\perp})_{\mathbb{R}}$ and $(X, \alpha) \in \mathfrak{M}_{L_2}^{\circ}$, we have
	$$
		\alpha( H^{2,0}(X) ) \in \Omega_{M^{\perp}} \cap \Omega^+_{h^{\perp}} =\Omega^+_{M^{\perp}}.
	$$ 
	
	By definition, every $(X, \iota) \in \tilde{\mathcal{M}}_{M, \mathcal{K}}$ carries an admissible marking for $(M, \mathcal{K})$.
	Moreover, if $\alpha, \beta : H^2(X, \mathbb{Z}) \to L_2$ are two admissible markings for $(M, \mathcal{K})$, 
	then we have $\alpha \circ \beta^{-1} \in \Gamma(\mathcal{K})$,
	where we define
	$$
		\Gamma(\mathcal{K}) =\{ \sigma \in \mon(L_2) ; \sigma \circ \iota_M = \iota_M \circ \sigma  \text{ and } \sigma(\mathcal{K})= \mathcal{K} \} =\{ \sigma \in \Gamma(M) ; \sigma(\mathcal{K})=\mathcal{K} \}.
	$$
	
	Let
	$$
		\Gamma_{M^{\perp}, \mathcal{K}} = \{ \sigma|_{M^{\perp}} \in O(M^{\perp}) ; \sigma \in \Gamma(\mathcal{K}) \}.
	$$ 
	By \cite[Proposition 10.2]{MR3519981}, $\Gamma_{M^{\perp}, \mathcal{K}}$ is a finite index subgroup of $O^+(M^{\perp})$.
	Therefore we obtain an orthogonal modular variety
	$$
		\mathcal{M}_{M, \mathcal{K}} = \Omega^+_{M^{\perp}}/\Gamma_{M^{\perp}, \mathcal{K}} .
	$$
	
	We define the period map $P_{M, \mathcal{K}} : \tilde{\mathcal{M}}_{M, \mathcal{K}} \to \mathcal{M}_{M, \mathcal{K}}$ by
	$$
		P_{M, \mathcal{K}}(X, \iota) = [\alpha(H^{2,0}(X))],
	$$
	where $\alpha : H^2(X, \mathbb{Z}) \to L_2$ is an admissible marking for $(M, \mathcal{K})$.
	
	Let 
	$
		\Delta(M^{\perp}) =\{ \delta \in M^{\perp} ; \delta^2=-2, \text{or } \delta^2=-10, (\delta, L_2)=2\mathbb{Z} \},
	$
	and set 
	$$
		\mathscr{D}_{M^{\perp}} = \bigcup_{\delta \in \Delta(M^{\perp})} H_{\delta} \subset \Omega_{M^{\perp}},
	$$
	where
	$
		H_{\delta} = \{ x \in \Omega_{M^{\perp}} ; (x, \delta)=0 \}.
	$
	By \cite[Lemma 7.7]{MR3519981}, $\mathscr{D}_{M^{\perp}}$ is locally finite in $\Omega_{M^{\perp}}$
	and is viewed as a reduced divisor on $\Omega_{M^{\perp}}$.
	Set $\bar{\mathscr{D}}_{M^{\perp}}= \mathscr{D}_{M^{\perp}}/\Gamma_{M^{\perp}, \mathcal{K}}$.
	Then $\bar{\mathscr{D}}_{M^{\perp}}$ is a reduced divisor on $\mathcal{M}_{M, \mathcal{K}}$.
	We set 
	$$
		\mathcal{M}^{\circ}_{M, \mathcal{K}} = \mathcal{M}_{M, \mathcal{K}} \setminus \bar{\mathscr{D}}_{M^{\perp}} \quad \text{  and  } \quad \Omega_{M^{\perp}}^{\circ} = \Omega^+_{M^{\perp}} \setminus \mathscr{D}_{M^{\perp}}.
	$$
	By \cite[Lemma 9.5 and Proposition 9.9]{MR3519981}, the image of the period map $P_{M, \mathcal{K}} : \tilde{\mathcal{M}}_{M, \mathcal{K}} \to \mathcal{M}_{M, \mathcal{K}}$ is $\mathcal{M}^{\circ}_{M, \mathcal{K}}$.

	Let $M_0$ be a primitive hyperbolic 2-elementary sublattice of $L_{K3}$.
	The orthogonal complement $M_0^{\perp_{L_{K3}}} = M_0^{\perp}$ of $M_0$ is considered in $L_{K3}$.
	Since $L_{K3}$ is unimodular and since $M_0$ is 2-elementary, the involution
	$$
		M_0 \oplus M_0^{\perp} \to M_0 \oplus M_0^{\perp}, \quad (m,n) \mapsto (m,-n) 
	$$
	extends uniquely to an involution $\iota_{M_0} \in O(L_{K3})$ by \cite[Corollary 1.5.2]{Nikulin1980IntegralSB}.
	
	Let $Y$ be a K3 surface and $\sigma : Y \to Y$ be an antisymplectic involution on $Y$.
	Set
	$$
		H^2(Y, \mathbb{Z})^{\sigma} = \left\{ x \in H^2(Y, \mathbb{Z}) ; \sigma^*x = x \right\}.
	$$
	Let $\alpha : H^2(Y, \mathbb{Z}) \to L_{K3}$ be an isometry.
	We call the pair $(Y, \sigma)$ a 2-elementary K3 surface of type $M_0$ if the restriction of $\alpha$ is an isometry from $H^2(Y, \mathbb{Z})^{\sigma}$ to $M_0$.
	
	Since $\sign(M_0^{\perp}) =(2, \rk M_0^{\perp}-2)$, $\Omega_{M_0^{\perp}}$ consists of two connected components.
	We fix a connected component $\Omega^+_{M_0^{\perp}}$ of $\Omega_{M_0^{\perp}}$.
	We obtain the orthogonal modular variety 
	$$
		\mathcal{M}_{M_0}=\Omega^+_{M_0^{\perp}}/O^+(M_0^{\perp})
	$$
	of dimension $20-\rk(M_0)$.
	
	We set
	$
		\Delta(M_0^{\perp}) =\{ d \in M_0^{\perp} ; d^2=-2 \}
	$
	and
	$
		\mathscr{D}_{M_0^{\perp}} = \bigcup_{d \in \Delta(M_0^{\perp})} H_d \subset \Omega^+_{M_0^{\perp}}.
	$
	Here $H_d =d^{\perp}$ is the divisor on $\Omega^+_{M_0^{\perp}}$ defined in the same way as before.
	By \cite[Proposition 1.9.]{MR2047658}, $\mathscr{D}_{M_0^{\perp}}$ is locally finite and is viewed as a reduced divisor on $\Omega^+_{M_0^{\perp}}$.
	Set 
	$$
		\bar{\mathscr{D}}_{M_0^{\perp}} = \mathscr{D}_{M_0^{\perp}} / O^+(M_0^{\perp}) \quad \text{ and } \quad \mathcal{M}^{\circ}_{M_0} = \mathcal{M}_{M_0} \setminus \bar{\mathscr{D}}_{M_0^{\perp}}.
	$$
	By \cite[Theorem 1.8.]{MR2047658}, $\mathcal{M}^{\circ}_{M_0}$ is a coarse moduli space of 2-elementary K3 surfaces of type $M_0$. 
	
	Recall that $L_2 = L_{K3} \oplus \mathbb{Z}e$.
	We consider the case
	$$
		M= \widetilde{M}_0 =M_0 \oplus \mathbb{Z}e .
	$$
	In this case, we define an involution $\iota_{\widetilde{M}_0} : L_2 \to L_2$ by
	$$
		 \iota_{\widetilde{M}_0}(x_0 + ae)=\iota_{M_0}(x_0) +ae \quad (x_0 \in L_{K3},  a \in \mathbb{Z}).
	$$
	Then $(\widetilde{M}_0, \iota_{\widetilde{M}_0})$ is an admissible sublattice of $L_2$.
	
	Let $\mathcal{K} \in \operatorname{KT}(\widetilde{M}_0)$.
	A hyperplane $H$ of $\widetilde{M}_{0,\mathbb{R}}$ is a face of $\mathcal{K}$ if $H \cap \partial \mathcal{K}$ contains an open subset of $H$.
	
	\begin{dfn}\label{d2-3-1-1}
		A K\"ahler-type chamber $\mathcal{K} \in \operatorname{KT}(\widetilde{M}_0)$ is natural if the hyperplane $M_{0, \mathbb{R}}$ is a face of $\mathcal{K}$.
	\end{dfn}
	
	Since $M_0^{\perp_{L_{K3}}}=\widetilde{M}_0^{\perp_{L_2}}$, we may identify $O(M_0^{\perp_{L_{K3}}})=O(\widetilde{M}_0^{\perp_{L_2}})$ and $\Omega^+_{M_0^{\perp}} = \Omega^+_{\widetilde{M}_0^{\perp}}$.
	
	\begin{thm}\label{t-3-1-1}
		If $\mathcal{K} \in \operatorname{KT}(\widetilde{M}_0)$ is natural, then $\Gamma_{\widetilde{M}_0^{\perp}, \mathcal{K}} = O^+(M_0^{\perp})$.
		In particular, the identity map $\Omega_{M_0^{\perp}} \to \Omega_{\widetilde{M}_0^{\perp}}$ induces an isomorphism of orthogonal modular varieties $\phi: \mathcal{M}_{M_0} \cong \mathcal{M}_{\widetilde{M}_0, \mathcal{K}}$ such that $\phi( \bar{\mathscr{D}}_{M_0^{\perp}} ) =\bar{\mathscr{D}}_{\widetilde{M}_0^{\perp}}$.
	\end{thm}
	
	\begin{proof}
		See \cite[Theorem 2.28 and Corollary 2.29]{I1} and \cite[Lemma 4.2]{I2}.
	\end{proof}
	
	%By \cite[Lemma 4.2]{I2}, we have
	%$$
	%	\phi( \bar{\mathscr{D}}_{M_0^{\perp}} ) =\bar{\mathscr{D}}_{\widetilde{M}_0^{\perp}}.
	%$$
	We identify $\mathcal{M}_{M_0}$ with $\mathcal{M}_{\widetilde{M}_0, \mathcal{K}}$ and $\bar{\mathscr{D}}_{M_0^{\perp}}$ with $\bar{\mathscr{D}}_{\widetilde{M}_0^{\perp}}$ via the isomorphism $\phi$.

\subsection{The invariant $\tau_{M, \mathcal{K}}$ and its properties.}\label{ss-1-3}
	
	We fix an admissible sublattice $M$ of $L_2$ and a K\"ahler-type chamber $\mathcal{K} \in \ktm$.
	Let $(X, \iota) \in \tilde{\mathcal{M}}_{M, \mathcal{K}}$ and let $h_X$ be an $\G$-invariant \K metric on $X$. 
	Let $\omega_X$ be the $\G$-invariant \K form attached to $h_X$.
	
	Let $\tau_{\iota}(\bar{\Omega}_X^1)$ be the equivariant analytic torsion of the cotangent bundle $\bar{\Omega}_X^1 =(\Omega_X^1, h_X)$ endowed with the hermitian metric induced from $h_X$.
	
	By \cite[Theorem 1]{MR2805992}, $X^{\iota}$ is a possibly disconnected compact complex surface.
	Let $X^{\iota} = \sqcup_i Z_i$ be the decomposition into the connected components.
	Let $\tau(\bar{\mathcal{O}}_{X^{\iota}})$ be the analytic torsion of the trivial bundle $\bar{\mathcal{O}}_{X^{\iota}}$ with respect to the canonical metric.
	Here $\tau(\bar{\mathcal{O}}_{X^{\iota}})$ is given by 
	$$
		\tau(\bar{\mathcal{O}}_{X^{\iota}}) = \prod_i \tau(\bar{\mathcal{O}}_{Z_i}),
	$$
	where $\tau(\bar{\mathcal{O}}_{Z_i})$ is the analytic torsion of the trivial bundle $\bar{\mathcal{O}}_{Z_i}$ with respect to the canonical metric.
	
	The volume of $(X, \omega_{X})$ is defined by
	$
		\vol(X, \omega_{X}) = \int_X \omega_{X}^4/4!.
	$
	Set $\omega_{X^{\iota}} = \omega_X|_{X^{\iota}}$.
	This is a \K form on $X^{\iota}$ attached to $h_{X^{\iota}} =h_X|_{X^{\iota}}$.
	%By \cite[Theorem 1]{MR2805992}, $X^{\iota}$ is a possibly disconnected compact complex surface.
	%Let $X^{\iota} = \sqcup_i Z_i$ be the decomposition into the connected components.
	We define the volume of $(X^{\iota}, \omega_{X^{\iota}})$ by
	$$
		\vol(X^{\iota}, \omega_{X^{\iota}}) =\prod_i \vol(Z_i, \omega_X|_{Z_i}) = \prod_i \int_{Z_i} \frac{(\omega_X|_{Z_i})^2}{2!}.
	$$
	The covolume of the lattice $\operatorname{Im}(H^1(X^{\iota}, \mathbb{Z}) \to H^1(X^{\iota}, \mathbb{R}))$ with respect to the $L^2$-metric induced from $h_X$ is denoted by $\vol_{L^2}(H^1(X^{\iota}, \mathbb{Z}), \omega_{X^{\iota}})$ .
	Namely, 
	$$
		\vol_{L^2}(H^1(X^{\iota}, \mathbb{Z}), \omega_{X^{\iota}}) = \det(\langle e_i, e_j \rangle_{L^2}),
	$$
	where $e_1, \dots ,e_{b_1(X^{\iota})}$ is an integral basis of $\operatorname{Im}(H^1(X^{\iota}, \mathbb{Z}) \to H^1(X^{\iota}, \mathbb{R}))$.
	
	We define a real-valued function $\varphi$ on $X$ by
	$$
		\varphi = \frac{\omega_{X}^4/4!}{\eta^2 \wedge \bar{\eta}^2} \frac{\|\eta^2\|_{L^2}^2}{\vol(X, \omega_{X})},
	$$
	where $\eta$ is a holomorphic symplectic 2-form on $X$.
	Obviously, $\varphi$ is independent of the choice of $\eta$.
	We define a positive number $A(X, \iota, h_X) \in \mathbb{R}_{>0}$ by
	\begin{align}\label{al2-3-1-2}
		A(X, \iota, h_X) =\exp \left[ \frac{1}{48} \int_{X^{\iota}} (\log \varphi) \varOmega \right],
	\end{align}
	where $\varOmega$ is a characteristic form on $X^{\iota}$ defined by
	\begin{align}\label{al2-3-1-3}
		\varOmega = c_1(TX^{\iota}, h_{X^{\iota}})^2 -8c_2(TX^{\iota}, h_{X^{\iota}}) -c_1(TX, h_X)|^2_{X^{\iota}} +3c_2(TX, h_X)|_{X^{\iota}}.
	\end{align}
	Here we denote by $c_i(TX, h_X)$, $c_i(TX^{\iota}, h_{X^{\iota}})$ the $i$-th Chern forms of the holomorphic hermitian vector bundles $(TX, h_X)$, $(TX^{\iota}, h_{X^{\iota}})$, respectively.
	Note that if $h_X$ is Ricci-flat, then we have $\varphi =1$ and $A(X, \iota, h_X) =1$.
	
	We set 
	$$
		t=\operatorname{Tr}(\iota^*|_{H^{1,1}(X)}).
	$$
	By \cite[Theorem 2]{MR2805992}, $t$ is an odd number with $-19 \leqq t \leqq 21$.
	By the definition of the admissible sublattice $(M, \iota_M)$, we have $t= \operatorname{Tr}(\iota_M)+2$.
	Therefore $t$ depends only on $(M, \iota_M)$ and is a constant function on $\tilde{\mathcal{M}}_{M, \mathcal{K}}$.
	
	%Let $\tau_{\iota}(\bar{\Omega}_X^1)$ be the equivariant analytic torsion of the cotangent bundle $\bar{\Omega}_X^1 =(\Omega_X^1, h_X)$ endowed with the hermitian metric induced from $h_X$,
	%and let $\tau(\bar{\mathcal{O}}_{X^{\iota}})$ be the analytic torsion of the trivial bundle $\bar{\mathcal{O}}_{X^{\iota}}$ with respect to the canonical metric.
	
	\begin{dfn}\label{d-3-1}
		We define a real number $\tau_{M, \mathcal{K}}(X, \iota)$ by
		\begin{align*}
			\tau_{M, \mathcal{K}}(X, \iota)=\tau_{\iota}(\bar{\Omega}_X^1) &\vol(X, \omega_{X})^{\frac{(t-1)(t-7)}{16}} A(X, \iota, h_X) \\
			&\cdot \tau(\bar{\mathcal{O}}_{X^{\iota}})^{-2} \vol(X^{\iota}, \omega_{X^{\iota}})^{-2} \vol_{L^2}(H^1(X^{\iota}, \mathbb{Z}), \omega_{X^{\iota}}).
		\end{align*}
	\end{dfn}
	
	%
	%Let $\X$ be a complex manifold with holomorphic involution $\iota : \X \to \X$,
	%let $S$ be a complex manifold,
	%and let $f : (\X, \iota) \to S$ be a family of $K3^{[2]}$-type manifolds with involution of type $(M, \mathcal{K})$.
	
	%The fixed locus of $\iota : \X \to \X$ is denoted by $\XX = \{ x \in \X ; \iota(x) =x \}$.
	%The restriction of $f : \X \to S$ to $\XX$ is also denoted by $f : \XX \to S$
	%and it is a family of smooth complex surfaces.
	%
	
	Let $f : (\X, \iota) \to S$ be a family of $K3^{[2]}$-type manifolds with involution of type $(M, \mathcal{K})$
	and let $h_{\xs}$ be an $\iota$-invariant fiberwise K\"ahler metric on $T\xs$.
	Let $h_{\xss}$ be the induced metric on $T\xss$.
	The $\iota$-invariant hermitian metric on $\Omega^1_{\xs}$ induced from $h_{\xs}$ is also denoted by $h_{\xs}$.
	
	We define the characteristic form $\omega_{H^{\cdot}(\xss)} \in A^{1,1}(S)$ by
	\begin{align}\label{al6-2-1-1}
		\omega_{H^{\cdot}(\xss)} = c_1(f_*\Omega^1_{\xss}, h_{L^2}) -c_1(R^1f_*\mathcal{O}_{\mathscr{X}^{\iota}}, h_{L^2}) -2c_1(f_*K_{\xss}, h_{L^2}).
	\end{align}
	By \cite[Lemma 3.13]{I1}, $\omega_{H^{\cdot}(\xss)}$ is independent of the choice of the $\G$-invariant fiberwise \K metric $h_{\xs}$.
	
	\begin{thm}\label{p-3-4}
		We define a real-valued function $\tau_{M, \mathcal{K}, \xs}$ on $S$ by 
		$$
			\tau_{M, \mathcal{K}, \xs}(s) =\tau_{M, \mathcal{K}}(X_s, \iota_s) \quad (s \in S).
		$$
		Then $\tau_{M, \mathcal{K}, \xs}$ is smooth and satisfies
		$$
			-dd^c \log \tau_{M, \mathcal{K}, \xs} = \frac{(t+1)(t+7)}{16} c_1(f_*K_{\xs}, h_{L^2}) +\omega_{H^{\cdot}(\xss)}.
		$$
	\end{thm}
	
	\begin{proof}
		See \cite[Theorem 3.12]{I1}.
	\end{proof}
	
	\begin{thm}\label{tA-1-1}
		For $(X, \iota) \in \tilde{\mathcal{M}}_{M, \mathcal{K}}$, the real number $\tau_{M, \mathcal{K}}(X, \iota)$ is independent of the choice of an $\iota$-invariant \K form.
		In particular, $\tau_{M, \mathcal{K}}(X, \iota)$ is an invariant of $(X, \iota)$.
	\end{thm}
	
	\begin{proof}
		See \cite[Theorem 3.14]{I1}.
	\end{proof}
	
	By \cite[Lemma 3.15]{I1}, $\tau_{M, \mathcal{K}}$ is viewed as a smooth real-valued function on $\mathcal{M}^{\circ}_{M, \mathcal{K}}$.
	Namely,
	$$
		\tau_{M,\mathcal{K}}(p) = \tau_{M,\mathcal{K}}(X, \iota) \quad ((X, \iota) \in P_{M, \mathcal{K}}^{-1}(p) )
	$$
	is independent of the choice of $(X, \iota) \in P_{M, \mathcal{K}}^{-1}(p)$.
	
	%Let $\omega_{\mathcal{M}_{M, \mathcal{K}}}$ be the orbifold \K form on $\mathcal{M}_{M, \mathcal{K}}$ induced from the \K form of the Bergman metric on the period domain $\Omega_{M^{\perp}}$.
	
	%In \cite[Lemma 3.16]{I1}, there exists a smooth $(1,1)$-form $\sigma_{M, \mathcal{K}}$ on $\mathcal{M}^{\circ}_{M, \mathcal{K}}$ such that for any $(X, \iota) \in \tilde{\mathcal{M}}_{M, \mathcal{K}}$ we have
	%$$
	%	P_{M, \mathcal{K}}^*\sigma_{M, \mathcal{K}} = c_1(\pi_*\Omega^1_{\XX / \operatorname{Def}(X, \iota)}, h_{L^2}) -c_1(R^1\pi_*\mathcal{O}_{\mathscr{X}^{\iota}}, h_{L^2}) -2c_1(\pi_*K_{\XX / \operatorname{Def}(X, \iota)}, h_{L^2}),
	%$$ 
	%where $P_{M, \mathcal{K}} : \operatorname{Def}(X, \iota) \to \mathcal{M}_{M, \mathcal{K}}$ is the period map of the Kuranishi family $\pi : (\X, \iota) \to \operatorname{Def}(X, \iota)$ of $(X, \iota)$.
	
	%\begin{thm}\label{t-0-A}
	%	The following equation of differential forms on $\mathcal{M}^{\circ}_{M, \mathcal{K}}$ holds:
	%	$$
	%		-dd^c \log \tau_{M, \mathcal{K}} = \frac{(t+1)(t+7)}{8} \omega_{\mathcal{M}_{M, \mathcal{K}}} +\sigma_{ M, \mathcal{K} }.
	%	$$
	%\end{thm}
	
	%\begin{proof}
	%	See \cite[Theorem 3.17]{I1}.
	%\end{proof}
	
	Let $\mathcal{M}_{M_0}^*$ be the Baily-Borel compactification of $\mathcal{M}_{M_0}$.
		
	\begin{thm}\label{p2-1-2-2}
		Let $C \subset \mathcal{M}^*_{M, \mathcal{K}}$ be an irreducible projective curve.
		We assume that $C$ is not contained in $\bar{\mathscr{D}}_{M^{\perp}} \cup (\mathcal{M}^*_{M, \mathcal{K}} \setminus \mathcal{M}_{M, \mathcal{K}})$.
		Let $p \in C \cap \bar{\mathscr{D}}_{M^{\perp}}$ be a smooth point of $C$.
		Choose a coordinate $(\D, s)$ on $C$ centered at $p$ such that $\D \cap \bar{\mathscr{D}}_{M^{\perp}} = \{p \}$ and $s(\D)$ is the unit disk.
		Then there exists a constant $a \in \mathbb{Q}$ such that the following identity holds:
		$$
			\log \tau_{M, \mathcal{K}} (s) = a \log |s|^2 + O(\log \log |s|^{-1}).
		$$
	\end{thm}
	
	\begin{proof}
		See \cite[Theorem 2.15]{I2}.
	\end{proof}

\subsection{The invariant $\tau_{M_0}$ and comparison between $\tau_{M,\mathcal{K}}$ and $\tau_{M_0}$}

	Let us recall the invariant $\tau_{M_0}$ of 2-elementary K3 surface of type $M_0$ introduced in \cite[Definition~5.1]{MR2047658}.
	Let  $(Y, \sigma)$ be 2-elementary K3 surface of type $M_0$.
	Choose a $\sigma$-invariant Ricci-flat \K metric $h_Y$ on $Y$and its associated \K form is denoted by $\omega_Y$.
	We define the volume of $(Y, \omega_Y)$ by
	$
		\vol(Y, \omega_{Y}) = \int_Y \omega_{Y}^2/2!.
	$
	The fixed locus of $\sigma : Y \to Y$ is denoted by $Y^{\sigma}$.
	Assume that $Y^{\sigma} \neq \emptyset$. 
	Let $Y^{\sigma} = \sqcup_i C_i$ be the decomposition into the connected components.
	We define the volume of $(Y^{\sigma}, \omega_{Y^{\sigma}})$ by
	$
		\vol(Y^{\sigma}, \omega_{Y^{\sigma}}) =\prod_i \vol(C_i, \omega_Y|_{C_i}) = \prod_i \int_{C_i} \omega_Y|_{C_i}.
	$
	%Let $\eta$ be a holomorphic 2-form on $Y$.
	%We define a positive number $A(Y, \sigma, h_Y) \in \mathbb{R}_{>0}$ by
	%\begin{align*}
	%	A(Y, \sigma, h_Y) =\exp \left[ \frac{1}{8} \int_{Y^{\sigma}} \log \left( \frac{\eta \wedge \bar{\eta}}{\omega_{Y}^2/2!} \frac{\vol(Y, \omega_{Y})}{\|\eta\|_{L^2}^2} \right) c_1(TY^{\sigma}, h_Y|_{Y^{\sigma}}) \right].
	%\end{align*}
	If $Y^{\sigma} = \emptyset$, we set $\vol(Y^{\sigma}, \omega_{Y^{\sigma}}) =1$.
	
	Let $\tau_{\sigma}(\bar{\mathcal{O}}_Y)$ be the equivariant analytic torsion of the trivial line bundle $\bar{\mathcal{O}}_Y$ with respect to the canonical metric,
	and let $\tau(\bar{\mathcal{O}}_{Y^{\sigma}})$ be the analytic torsion of the trivial line bundle $\bar{\mathcal{O}}_{X^{\iota}}$ with respect to the canonical metric.
	If $Y^{\sigma} = \emptyset$, we set $\tau(\bar{\mathcal{O}}_{Y^{\sigma}})=1$.
	
	\begin{dfn}\label{d7-3-3-1}
		Let  $(Y, \sigma)$ be 2-elementary K3 surface of type $M_0$
		and let $h_Y$ be a $\sigma$-invariant Ricci-flat \K metric on $Y$.
		We define a real number $\tau_{M_0}(Y, \sigma)$ by
		\begin{align*}
			\tau_{M_0}(Y, \sigma)=\tau_{\sigma}(\bar{\mathcal{O}}_Y) \vol(Y, \omega_{Y})^{\frac{14 - \rk(M_0)}{4}}
			\tau(\bar{\mathcal{O}}_{Y^{\sigma}}) \vol(Y^{\sigma}, \omega_{Y^{\sigma}}).
		\end{align*}
	\end{dfn}
	
	By \cite[Theorem~5.7]{MR2047658}, $\tau_{M_0}(Y, \sigma)$ is independent of the choice of a $\sigma$-invariant Ricci-flat \K metric $h_Y$
	and is an invariant of $(Y, \sigma)$.
	
	If we set $\widetilde{M}_0=M_0 \oplus \mathbb{Z}e$, then $\widetilde{M}_0$ is an admissible sublattice of $L_2$.
	Let $\mathcal{K} \in \operatorname{KT}(\widetilde{M}_0)$ be a natural K\"ahler-type chamber.
	We consider the case where $r=\rk(M_0) \leqq 17$ and the divisor $\bar{\mathscr{D}}_{M_0^{\perp}}$ is irreducible.
	
	\begin{thm}\label{t2-1-4-1}
		Suppose that $\rk(M_0) \leqq 17$ and the divisor $\bar{\mathscr{D}}_{M_0^{\perp}}$ is irreducible.
		Then there exists a positive constant $C_{M_0}>0$ depending only on $M_0$ such that, for any 2-elementary K3 surface $(Y, \sigma)$ of type $M_0$, the following identity holds:
		$$
			\tau_{\widetilde{M}_0, \mathcal{K}}(Y^{[2]}, \sigma^{[2]})=C\tau_{M_0}(Y, \sigma)^{-2(r-9)}.
		$$
	\end{thm}
	
	\begin{proof}
		See \cite[Theorem 4.6]{I2}.
	\end{proof}
		
	In this case, it follows from \cite[Theorem 0.1]{MR4177283} that $\tau_{\widetilde{M}_0, \mathcal{K}}(Y^{[2]}, \sigma^{[2]})$  is expressed as the Petersson norm of a certain automorphic form on a bounded symmetric domain of type IV and a certain Siegel modular form.

\subsection{The $L^2$-metric on direct image sheaves}

	%Let us show the fundamental properties of the $L^2$-metric on direct image sheaves.
	
	Let $\mathscr{M}$ and $B$ be complex manifolds of dimension $m+n$ and $n$, respectively.
	Let $\pi : \mathscr{M} \to B$ be a proper holomorphic submersion.
	Set $M_b = \pi^{-1}(b)$ for $b \in B$.
	Let $h_{\mathscr{M} / B}$ be a fiberwise \K metric on the relative tangent bundle $T\mathscr{M} / B$.
	The hermitian metric on $\Omega^p_{\mathscr{M} / B}$ induced from $h_{\mathscr{M} / B}$ is also denoted by $h_{\mathscr{M} / B}$.
	
	Since $\dim H^q(M_b, \Omega^p_{M_b})$ is independent of $b \in B$,
	the direct image sheaf $R^q\pi_*\Omega^p_{\mathscr{M} / B}$ is locally free.
	We regard $R^q\pi_*\Omega^p_{\mathscr{M} / B}$ as a holomorphic vector bundle on $B$.
	By Hodge theory, $R^q\pi_*\Omega^p_{\mathscr{M} / B}$ is equipped with the hermitian metric.
	This is called the $L^2$-metric and is denoted by $h_{L^2}$.
	The dual metric on the dual $\left( R^{q}\pi_*\Omega^{p}_{\mathscr{M} / B} \right)^{*}$ is denoted by $h_{L^2}^*$.
	By the Serre duality, there exists an isomorphism of holomorphic vector bundles
	$$
		\Phi : R^q\pi_*\Omega^p_{\mathscr{M} / B} \to \left( R^{m-q}\pi_*\Omega^{m-p}_{\mathscr{M} / B} \right)^{*}.
	$$ 
	The following two lemmas are classical. 
	
	\begin{lem}\label{l-1-4-1}
		The following identity holds:
		$$
			c_1(R^q\pi_*\Omega^p_{\mathscr{M} / B}, h_{L^2}) =-c_1(R^{m-q}\pi_*\Omega^{m-p}_{\mathscr{M} / B}, h_{L^2}) .
		$$
	\end{lem}

	\begin{lem}\label{l-1-4-4}
		Let $M$ be a compact \K manifold of dimension $4$.
		%Assume that $m = \dim M =4$.
		Then the $L^2$-inner product $h_{L^2}$ on the cohomology $H^1(M, \Omega_M^p)$ $(p=1, 3)$ depends only on the choice of the K\"ahler class
		and is independent of the choice of a K\"ahler metric itself.
	\end{lem}

\section{Comparison between $\tau_{M, \mathcal{K}}$ and $\tau_{BCOV}$}\label{s-2}

	In this section, we give a comparison theorem between our invariant $\tau_{M, \mathcal{K}}$ and the BCOV invariant $\tau_{BCOV}$ of Calabi-Yau 4-folds of Camere-Garbagnati-Mongardi.
	The crucial step of the proof is to show that $\tau_{M, \mathcal{K}}$ satisfies the same curvature equation as that of $\tau_{BCOV}$.

\subsection{The BCOV invariant of Calabi-Yau manifolds}\label{ss-2-0}
	
	A compact \K manifold $Z$ of dimension $m$ is a Calabi-Yau manifold if $K_Z \cong \mathcal{O}_Z$ and 
	$
		H^q(Z, \mathcal{O}_Z) =0 
	$ 
	for $0<q<m$.
	Let $h_Z$ be a \K metric on $Z$.
	Its associated \K form is denoted by $\omega_Z$.
	The hermitian metric on $\Omega^p_Z$ induced from $h_Z$ is also denoted by $h_Z$.
	Set $\overline{\Omega}^p_Z =(\Omega^p_Z, h_Z)$.
	Let $\tau(\overline{\Omega}^p_Z)$ be the analytic torsion of the holomorphic hermitian vector bundle $\overline{\Omega}^p_Z$.
	Define the BCOV torsion $\mathcal{T}_{BCOV}(Z, \omega_Z)$ by
	$$
		\mathcal{T}_{BCOV}(Z, \omega_Z) = \prod_{p \geqq 0} \tau(\overline{\Omega}^p_Z)^{(-1)^p p}.
	$$
	
	Let $\eta$ be a generator of $H^0(Z, \Omega^m_Z) \cong \mathbb{C}$.
	We set
	$$
		A(Z, \omega_Z) = \exp \left\{ -\frac{1}{12} \int_Z (\log \varphi) c_m(TZ, h_Z) \right\},
	$$
	where $\varphi$ is a smooth function on $Z$ defined by
	$$
		\varphi = \frac{i^{m^2} \eta \wedge \bar{\eta}}{\omega_Z^m / m !} \frac{1}{ \| \eta \|^2_{L^2} }.
	$$
	Note that if $h_Z$ is a Ricci-flat \K metric, then we have
	$$
		A(Z, \omega_Z) = \vol(Z, \omega_{Z})^{\frac{\chi(Z)}{12}},
	$$
	where $\chi(Z)$ is a topological Euler characteristic of $Z$.
	
	The covolume of the lattice $\operatorname{Im}(H^k(Z, \mathbb{Z}) \to H^k(Z, \mathbb{R}))$ with respect to the $L^2$-metric induced from $h_Z$ is denoted by $\vol_{L^2}(H^k(Z, \mathbb{Z}), \omega_{Z})$.
	We set
	$$
		B(Z, \omega_Z) = \prod_{k=1}^{2m} \vol_{L^2}(H^k(Z, \mathbb{Z}), \omega_{Z})^{ (-1)^{k+1} k/2 }.
	$$
	
	\begin{dfn}\label{d3-2-1-1}
		Let $Z$ be a Calabi-Yau $m$-fold
		and Let $\omega_Z$ be a \K form on $Z$.
		The BCOV invariant of $Z$ is defined by 
		$$
			\tau_{BCOV}(Z) = \frac{ A(Z, \omega_Z) }{ B(Z, \omega_Z) }\mathcal{T}_{BCOV}(Z, \omega_Z).
		$$
	\end{dfn}
	
	\begin{prop}\label{p3-2-1-2}
		The BCOV invariant $\tau_{BCOV}(Z)$ is independent of the choice of $\omega_Z$.
	\end{prop}
	
	\begin{proof}
		See \cite[Proposition 5.8]{MR4255041}.
	\end{proof}

	%Throughout this section, we fix an admissible sublattice $M$ of $L_2$ and a K\"ahler-type chamber $\mathcal{K} \in \ktm$.
	%Let $(X, \iota)$ be a manifold of $K3^{[2]}$-type with involution of type $(M, \mathcal{K})$.
	%The blowup of $X$ along the fixed locus $X^{\iota}$ is denoted by $\widetilde{X}$.
	%The involution $\iota$ induces an involution $\tilde{\iota}$ of $\widetilde{X}$
	%and its quotient is denoted by $Z = \widetilde{X} / \tilde{\iota}$.
	%Camere-Garbagnati-Mongardi \cite[Theorem 3.6]{MR3928256} proved that $Z$ is a Calabi-Yau 4-fold and the morphism $Z \to X/\iota$ induced from the blowup $\widetilde{X} \to X$ is a crepant resolution.

\subsection{Calabi-Yau 4-folds of Camere-Garbagnati-Mongardi and their BCOV invariants}\label{ss-2-1}

	Throughout this section, we fix an admissible sublattice $M$ of $L_2$ and a K\"ahler-type chamber $\mathcal{K} \in \ktm$.
	
	Let $(X, \iota)$ be a manifold of $K3^{[2]}$-type with involution of type $(M, \mathcal{K})$.
	The fixed locus is denoted by $X^{\iota}$.
	The blowup of $X$ along the fixed locus $X^{\iota}$ is denoted by $\widetilde{X}$.
	The involution $\iota : X \to X$ induces an involution $\tilde{\iota} : \widetilde{X} \to \widetilde{X}$ of $\widetilde{X}$.
	By Camere-Garbagnati-Mongardi \cite[Theorem 3.6]{MR3928256}, the quotient $\widetilde{X} / \tilde{\iota}$ is a Calabi-Yau 4-fold and we set
	$$
		Z = \widetilde{X} / \tilde{\iota}.
	$$
	
	\begin{dfn}
		The manifold $Z$ is called a Calabi-Yau manifold of Camere-Garbagnati-Mongardi.
	\end{dfn}
	
	Let $f : (\X, \iota) \to S$ be a family of $K3^{[2]}$-type manifolds with involution of type $(M, \mathcal{K})$.
	Let $X_s =f^{-1}(s)$ be the fiber at $s \in S$.
	The fixed locus of $\iota$ is denoted by $\XX$
	and the blowup of $\X$ along $\XX$ is denoted by $\tau : \widetilde{\X} =Bl_{\XX}(\X) \to \X$.
	The involution $\iota$ induces an involution $\tilde{\iota}$ of $\widetilde{\X}$
	and its quotient is denoted by $\Z = \widetilde{\X} / \tilde{\iota}$.
	The morphism induced from $\widetilde{\X} \xrightarrow{\tau} \X \xrightarrow{f} S$ is denoted by $h : \Z \to S$.
	By \cite[Theorem 3.6]{MR3928256}, each fiber $Z_s = h^{-1}(s)$ is a Calabi-Yau 4-fold and $h : \Z \to S$ is a family of Calabi-Yau 4-folds.

	Let $h_{\xs}$ be an $\iota$-invariant hermitian metric on the relative tangent bundle $T\xs$ which is fiberwise K\"ahler.
	Since $\dim H^q(X_s, \Omega^p_{X_s})$ is independent of $s \in S$, $R^qf_*\Omega^p_{\xs}$ is a locally free sheaf and we regard it as a holomorphic vector bundle on $S$.
	By Hodge theory, $R^qf_*\Omega^p_{\xs}$ is equipped with the $\iota$-invariant hermitian metric.
	This metric is called the $L^2$-metric and is denoted by $h_{L^2}$.
	
	Let $h_{\zs}$ be a hermitian metric on $T\zs$ which is fiberwise K\"ahler.
	In the same manner as above, the holomorphic vector bundle $R^qh_*\Omega^p_{\zs}$ is equipped with the $L^2$ metric,
	which is denoted by $h_{L^2}$.
	
	%The BCOV invariant constructed in \cite{MR4255041} is denoted by $\tau_{BCOV}$.
	Define a function $\tau_{BCOV, \zs}$ on $S$ by
	$$
		\tau_{BCOV, \zs}(s) = \tau_{BCOV}(Z_s) \quad (s \in S). 
	$$
	By \cite[Proposition 5.10]{MR4255041}, $\tau_{BCOV, \zs}$ satisfies the following curvature equation on $S$:
	\begin{align}\label{al-2-1-1}
		dd^c \log \tau_{BCOV, \zs} = \sum^{8}_{k=0} (-1)^k \omega_{H^k} -\frac{\chi}{12} \omega_{WP},
	\end{align}
	where 
	\begin{align}\label{al-2-1-2}
		\omega_{H^k} = \sum_{p+q=k} p\cpq{p}{q} \quad \text{and} \quad \omega_{WP} = c_1(h_*K_{\zs}, h_{L^2}).
	\end{align}
	
	In the next subsection, we show that $\tau_{M, \mathcal{K}, \xs}$ satisfies the same equation (\ref{al-2-1-1}). 
	
	\begin{thm}\label{t-2-1-3}
		The following equation on $S$ holds:
		$$
			dd^c \log \tau_{M, \mathcal{K}, \xs} = \sum^{8}_{k=0} (-1)^k \omega_{H^k} -\frac{\chi}{12} \omega_{WP}.
		$$
	\end{thm}

\subsection{Proof of Theorem \ref{t-2-1-3}}\label{ss-2-2}

	\begin{lem}\label{l-2-2-1}
		The following identity holds:
		$$
			\sum^8_{k=0} (-1)^k \omega_{H^k} = 2\sum^3_{q=1} (-1)^q c_1( R^qh_* \Omega^1_{\zs}, h_{L^2} ) -4 \left\{ c_1( h_* \mathcal{O}_{\Z}, h_{L^2} ) +c_1( R^4h_*\mathcal{O}_{\Z}, h_{L^2} ) \right\}.
		$$
	\end{lem}

	\begin{proof}
		Since each fiber $Z_s$ is a Calabi-Yau 4-fold, we have
		\begin{align}\label{al-2-2-2}
			R^qh_*\Omega^p_{\zs} =0
		\end{align}
		for $(p,q) = (1, 0), (2, 0), (3, 0), (0, 1), (0, 2), (0, 3), (1, 4), (2, 4), (3, 4), (4, 1), (4, 2), (4, 3)$.
		By (\ref{al-2-1-2}) and (\ref{al-2-2-2}), we have
		\begin{align}\label{al-2-2-3}
		\begin{aligned}
			\omega_{H^0} &=0, \quad
			\omega_{H^1} = 0, \quad
			\omega_{H^2} = \cpq{1}{1} , \\
			\omega_{H^3} &= \cpq{1}{2} +2\cpq{2}{1}.
		\end{aligned}
		\end{align}
		By the Serre duality, we have $\cpq{2}{2} = -\cpq{2}{2}$.
		Therefore
		\begin{align}\label{al3-2-2-1}
			\cpq{2}{2} =0.
		\end{align}
		By (\ref{al-2-1-2}), (\ref{al-2-2-2}), (\ref{al3-2-2-1}) and by the Serre duality, we have 
		\begin{align}\label{al-2-2-4}
		\begin{aligned}
			\omega_{H^4} &= \cpq{1}{3} +2\cpq{2}{2} +3\cpq{3}{1} +4\cpq{4}{0} \\
						&= -2\cpq{1}{3} -4 \cpq{0}{4},\\
			\omega_{H^5} &= 2\cpq{2}{3} +3\cpq{3}{2}\\
						&= -2\cpq{2}{1} -3\cpq{1}{2}, \\
			\omega_{H^6} &= 3\cpq{3}{3}
						= -3\cpq{1}{1}, \\
			\omega_{H^7} &= 0, \quad
			\omega_{H^8} = 4\cpq{4}{4} =-4c_1(h_* \mathcal{O}_{\Z}, h_{L^2}).
		\end{aligned}
		\end{align}
		By (\ref{al-2-2-3}) and (\ref{al-2-2-4}), we obtain the desired result.
	\end{proof}

	Let $X$ be a fiber of $f : \X \to S$.
	The blowup of $X$ along the fixed locus $X^{\iota}$ is denoted by $\tau : \widetilde{X} \to X$.
	The involution of $\widetilde{X}$ induced from $\iota$ is denoted by $\tilde{\iota}$.
	Then the quotient $Z := \widetilde{X} / \tilde{\iota}$ is a fiber of $h : \Z \to S$.
	Let $E$ be the exceptional divisor and let $j : E \hookrightarrow \widetilde{X}$ be the inclusion.
	Let $H^q(X, \Omega^p_X)_{\pm}$ and $H^q(\widetilde{X}, \Omega^p_{\widetilde{X}})_{\pm}$ be the $(\pm 1)$-eigenspaces of $H^q(X, \Omega^p_X)$ and $H^q(\widetilde{X}, \Omega^p_{\widetilde{X}})$ with respect to the $\iota$ and $\tilde{\iota}$-action, respectively.
	By \cite[Theorem 7.31]{MR2451566}, we have
	\begin{align}\label{al2-2-2-1}
		\tau^* \oplus j_* (\tau|_E)^* :  H^q(X, \Omega^p_X) \oplus H^{q-1}(X^{\iota}, \Omega^{p-1}_{X^{\iota}}) \cong H^q(\widetilde{X}, \Omega^p_{\widetilde{X}}),
	\end{align}
	where $j_* : H^{q-1}(E, \Omega^{p-1}_{E}) \to H^q(\widetilde{X}, \Omega^p_{\widetilde{X}})$ is the Gysin map.
	Since $\tau : \widetilde{X} \to X$ is a $\G$-equivariant map with respect to $\iota$ and $\tilde{\iota}$ 
	and since $\G$ acts trivially on $X^{\iota}$ and $E$,
	the isomorphism (\ref{al2-2-2-1}) induces the isomorphism of eigenspaces
	\begin{align}\label{al2-2-2-2}
		 H^q(X, \Omega^p_X)_{\pm} \oplus H^{q-1}(X^{\iota}, \Omega^{p-1}_{X^{\iota}}) \cong H^q(\widetilde{X}, \Omega^p_{\widetilde{X}})_{\pm}.
	\end{align}
	Note that the pullback of the projection $\widetilde{X} \to Z$ induces an isomorphism
	\begin{align}\label{al2-2-2-3}
		 H^q(\widetilde{X}, \Omega^p_{\widetilde{X}})_{+} \cong H^q(Z, \Omega^p_Z).
	\end{align}
	By (\ref{al2-2-2-2}) and (\ref{al2-2-2-3}), we have the canonical isomorphism
	\begin{align}\label{al2-2-2-4}
		 H^q(X, \Omega^p_X)_{+} \oplus H^{q-1}(X^{\iota}, \Omega^{p-1}_{X^{\iota}}) \cong H^q(Z, \Omega^p_Z).
	\end{align}
	(See \cite[proof of Theorem 4.1]{MR3928256}.)

	Let $\left( R^qf_*\Omega^p_{\xs} \right)_{\pm}$ be the $(\pm 1)$-eigenbundles of $R^qf_*\Omega^p_{\xs}$.
	By (\ref{al2-2-2-4}), we have an isomorphism of holomorphic vector bundles
	$$
		\phi : \left( R^qf_*\Omega^p_{\xs} \right)_{+} \oplus R^{q-1}f_*\Omega^{p-1}_{\xss} \cong R^qh_*\Omega^p_{\zs}.
	$$
	Similarly, we have an isomorphism of holomorphic vector bundles
	$$
		\psi : \left( R^pf_*\Omega^q_{\xs} \right)_{+} \oplus R^{p-1}f_*\Omega^{q-1}_{\xss} \cong R^ph_*\Omega^q_{\zs}.
	$$

	In the category of $C^{\infty}$ complex vector bundles on $S$, we can identify $R^qf_*\Omega^p_{\xs}$ with the $C^{\infty}$ complex vector bundle whose fiber is the space of harmonic $(p,q)$-forms.
	Similarly, $R^{q-1}f_*\Omega^{p-1}_{\xss}$ and $R^qh_*\Omega^p_{\zs}$ can be identified with $C^{\infty}$ complex vector bundles whose fiber is the space of harmonic forms.
	Let $c : \left( R^qf_*\Omega^p_{\xs} \right)_{+} \to \left( R^pf_*\Omega^q_{\xs} \right)_{+}$, $c : R^{q-1}f_*\Omega^{p-1}_{\xss} \to R^{p-1}f_*\Omega^{q-1}_{\xss}$ and $c: R^qh_*\Omega^p_{\zs} \to R^ph_*\Omega^q_{\zs}$ be the homomorphisms of $C^{\infty}$ complex vector bundles induced from the complex conjugation of harmonic forms.
	Then the following diagram of $C^{\infty}$ complex vector bundles commutes:
	\begin{align}\label{al-2-2-5}
		\xymatrix{
			\left( R^qf_*\Omega^p_{\xs} \right)_{+} \oplus R^{q-1}f_*\Omega^{p-1}_{\xss}  \ar[r]^{\hspace{50pt}  \phi} \ar[d]_{c} & R^qh_*\Omega^p_{\zs} \ar[d]^{c}   \\
			 \left( R^pf_*\Omega^q_{\xs} \right)_{+} \oplus R^{p-1}f_*\Omega^{q-1}_{\xss} \ar[r]^{\hspace{50pt}  \psi} & R^ph_*\Omega^q_{\zs} .
		}
	\end{align}

	\begin{lem}\label{l-2-2-5}
		The following identity of $(1,1)$-forms on $S$ holds:
		\begin{align*}
			&\quad \cpq{p}{q} -\cpq{q}{p} \\
			&= c_1 \left( \left( R^qf_*\Omega^p_{\xs} \right)_{+}, h_{L^2} \right) -c_1 \left(  \left( R^pf_*\Omega^q_{\xs} \right)_{+}, h_{L^2} \right) \\
			&\quad +c_1( R^{q-1}f_*\Omega^{p-1}_{\xss} , h_{L^2}) -c_1( R^{p-1}f_*\Omega^{q-1}_{\xss} , h_{L^2}).
		\end{align*}
	\end{lem}
	
	\begin{proof}
		Let $\alpha_1, \dots, \alpha_r$ be a local holomorphic frame of $\left( R^qf_*\Omega^p_{\xs} \right)_{+} \oplus R^{q-1}f_*\Omega^{p-1}_{\xss}$ and let $\beta_1, \dots, \beta_r$ be a local holomorphic frame of $\left( R^pf_*\Omega^q_{\xs} \right)_{+} \oplus R^{p-1}f_*\Omega^{q-1}_{\xss}$.
		We define smooth functions $\tilde{c}_1(\phi)$ and $\tilde{c}_1(\psi)$ by
		\begin{align}\label{al-2-2-6}
			\tilde{c}_1(\phi) = \log \frac{\det (\langle \phi( \alpha_i), \phi( \alpha_j ) \rangle_{L^2}) }{\det (\langle  \alpha_i,  \alpha_j  \rangle_{L^2})} \quad \text{and} \quad \tilde{c}_1(\psi) = \log \frac{\det (\langle \phi( \beta_i), \phi( \beta_j ) \rangle_{L^2}) }{\det (\langle  \beta_i,  \beta_j  \rangle_{L^2})}
		\end{align}
		Then $\tilde{c}_1(\phi)$ and $\tilde{c}_1(\psi)$ are independent of the choices of holomorphic frames $\alpha_1, \dots, \alpha_r$ and $\beta_1, \dots, \beta_r$
		and are globally defined on $S$.
		Note that $\tilde{c}_1(\phi)$ and $\tilde{c}_1(\psi)$ are the Bott-Chern secondary forms such that 
		\begin{align}\label{al3-2-2-2}
		\begin{aligned}
			\cpq{p}{q} -c_1 \left( \left( R^qf_*\Omega^p_{\xs} \right)_{+}, h_{L^2} \right) -c_1( R^{q-1}f_*\Omega^{p-1}_{\xss} ,h_{L^2}) &= -dd^c \tilde{c}_1(\phi) \\
			\cpq{q}{p} -c_1 \left( \left( R^pf_*\Omega^q_{\xs} \right)_{+}, h_{L^2} \right) -c_1( R^{p-1}f_*\Omega^{q-1}_{\xss} ,h_{L^2}) &= -dd^c \tilde{c}_1(\psi) 
		\end{aligned}
		\end{align}
		
		Note that $c(\alpha_1), \dots, c(\alpha_r)$ are local $C^{\infty}$ sections of $\left( R^pf_*\Omega^q_{\xs} \right)_{+} \oplus R^{p-1}f_*\Omega^{q-1}_{\xss}$ and are not holomorphic sections.
		Similarly, $c( \phi(\alpha_1) ), \dots, c( \phi(\alpha_r) )$ are local $C^{\infty}$ sections of  $R^ph_*\Omega^q_{\zs}$ and are not holomorphic sections.
		Let $C = (c_{ij})$ be a $r \times r$-matrix defined by 
		\begin{align}\label{al-2-2-7}
			c(\alpha_i) = \sum^r_{j=1} c_{ij} \beta_j.
		\end{align}
		Then $c_{ij}$ are $C^{\infty}$ functions.
		By the above commutative diagram (\ref{al-2-2-5}), we have
		\begin{align}\label{al-2-2-8}
			c( \phi( \alpha_i ) ) = \psi( c( \alpha_i ) ) = \sum^r_{j=1} c_{ij} \psi( \beta_j ).
		\end{align}
		By (\ref{al-2-2-7}), we have
		$$
			{}^t (\langle \alpha_i, \alpha_j \rangle_{L^2}) = ( \overline{ \langle \alpha_i, \alpha_j \rangle_{L^2} }) = (\langle c(\alpha_i), c(\alpha_j) \rangle_{L^2}) = C \cdot (\langle  \beta_i,  \beta_j  \rangle_{L^2}) \cdot {}^t\overline{C}.
		$$
		Therefore
		\begin{align}\label{al-2-2-9}
			\det (\langle \alpha_i, \alpha_j \rangle_{L^2}) =  | \det C |^2  \cdot  \det (\langle  \beta_i,  \beta_j  \rangle_{L^2}).
		\end{align}
		Similarly, by (\ref{al-2-2-8}), we have 
		\begin{align}\label{al-2-2-10}
			\det (\langle \phi(\alpha_i), \phi(\alpha_j) \rangle_{L^2}) =  | \det C |^2   \cdot  \det (\langle  \psi(\beta_i),  \psi(\beta_j)  \rangle_{L^2}) .
		\end{align}
		By (\ref{al-2-2-6}), (\ref{al-2-2-9}) and (\ref{al-2-2-10}), we have
		\begin{align}\label{al2-2-2-5}
		\begin{aligned}
			\tilde{c}_1(\phi) &= \log \frac{\det (\langle \phi( \alpha_i), \phi( \alpha_j ) \rangle_{L^2}) }{\det (\langle  \alpha_i,  \alpha_j  \rangle_{L^2})} \\
			&= \log \frac{\det (\langle \phi( \beta_i), \phi( \beta_j ) \rangle_{L^2}) }{\det (\langle  \beta_i,  \beta_j  \rangle_{L^2})} 
			= \tilde{c}_1(\psi) .
		\end{aligned}
		\end{align}
		By (\ref{al3-2-2-2}) and (\ref{al2-2-2-5}), we have
		\begin{align*}
			&\quad \cpq{p}{q} -c_1 \left( \left( R^qf_*\Omega^p_{\xs} \right)_{+}, h_{L^2} \right) -c_1( R^{q-1}f_*\Omega^{p-1}_{\xss} ,h_{L^2}) \\
			&= \Bigl. -dd^c \tilde{c}_1(\phi) \Bigr. 
			= \Bigl. -dd^c \tilde{c}_1(\psi) \Bigr. \\
			&= \cpq{q}{p} -c_1 \left( \left( R^pf_*\Omega^q_{\xs} \right)_{+}, h_{L^2} \right) -c_1( R^{p-1}f_*\Omega^{q-1}_{\xss} ,h_{L^2}),
		\end{align*}
		which completes the proof.
	\end{proof}

	\begin{lem}\label{l-2-2-11}
		The following identities hold:
		\begin{align*}
			\cpq{1}{1} &= c_1( R^2h_*\mathbb{C} \otimes \mathcal{O}_S, h_{L^2} ), \\ 
			\cpq{1}{2} &= \frac{1}{2} \left\{ c_1( R^{1}f_*\mathcal{O}_{\XX} , h_{L^2}) -c_1( f_*\Omega^{1}_{\xss} , h_{L^2}) \right\} +\frac{1}{2} c_1( R^3h_*\mathbb{C} \otimes \mathcal{O}_S, h_{L^2} ).
		\end{align*}
	\end{lem}

	\begin{proof}
		By (\ref{al-2-2-2}), the first identity holds.
		Since $R^3h_*\mathbb{C} \otimes \mathcal{O}_S$ is endowed with the Hodge filtration
		$
			R^3h_*\mathbb{C} \otimes \mathcal{O}_S = \mathcal{F}^0 \supset \mathcal{F}^1 \supset \mathcal{F}^2 \supset \mathcal{F}^3 \supset \mathcal{F}^4=0 
		$
		with
		$
			 \mathcal{F}^p /  \mathcal{F}^{p+1} = R^{3-p}h_* \Omega^p_{\zs},
		$
		we deduce from (\ref{al-2-2-2}) that
		\begin{align}\label{al2-2-2-6}
			 \mathcal{F}^1 = \mathcal{F}^0=R^3h_*\mathbb{C} \otimes \mathcal{O}_S, \quad  \mathcal{F}^3 = \mathcal{F}^4=0
		\end{align}
		By the short exact sequence
		$
			0 \to \mathcal{F}^{p+1} \to \mathcal{F}^p \to R^{3-p}h_* \Omega^p_{\zs} \to 0,
		$
		we have
		\begin{align}\label{al2-2-2-7}
		\begin{aligned}
			c_1(\mathcal{F}^1, h_{L^2}) -c_1(\mathcal{F}^2, h_{L^2}) &= c_1(R^2h_* \Omega^1_{\zs}, h_{L^2}), \\
			 c_1(\mathcal{F}^2, h_{L^2}) -c_1(\mathcal{F}^3, h_{L^2}) &= c_1(R^1h_* \Omega^2_{\zs}, h_{L^2}) .
		\end{aligned}
		\end{align}
		By (\ref{al2-2-2-6}) and (\ref{al2-2-2-7}),
		we have
		\begin{align}\label{al-2-2-12}
			c_1( R^3h_*\mathbb{C} \otimes \mathcal{O}_S, h_{L^2} ) =\cpq{1}{2} + \cpq{2}{1}.
		\end{align}
		Since the fibers of $f : \X \to S$ are manifolds of $K3^{[2]}$-type, we have
		\begin{align}\label{al-2-2-13}
			R^qf_*\Omega^p_{\xs} =0
		\end{align}
		for $p, q$ with $p+q= 1, 3, 5 \text{ or } 7$.
		Then we get
		\begin{align*}
			&\quad \cpq{1}{2} \\
			&= \frac{1}{2} \left\{ \cpq{1}{2} -\cpq{2}{1} \right\} +\frac{1}{2} c_1( R^3h_*\mathbb{C} \otimes \mathcal{O}_S, h_{L^2} ) \\
			&= \frac{1}{2} \left\{ c_1( R^2f_*\Omega^1_{\xs}, h_{L^2} ) -c_1( R^1f_*\Omega^2_{\xs}, h_{L^2} ) \right\} \\
			&\quad +\frac{1}{2} \left\{ c_1( R^{1}f_*\mathcal{O}_{\XX} , h_{L^2}) -c_1( f_*\Omega^{1}_{\xss} , h_{L^2}) \right\} 
			+\frac{1}{2} c_1( R^3h_*\mathbb{C} \otimes \mathcal{O}_S, h_{L^2} ) \\
			&=\frac{1}{2} \left\{ c_1( R^{1}f_*\mathcal{O}_{\XX} , h_{L^2}) -c_1( f_*\Omega^{1}_{\xss} , h_{L^2}) \right\} 
			+\frac{1}{2} c_1( R^3h_*\mathbb{C} \otimes \mathcal{O}_S, h_{L^2} ),
		\end{align*}
		where the first equality follows from (\ref{al-2-2-12}),
		the second equality follows from Lemma \ref{l-2-2-5},
		and the last equality follows from (\ref{al-2-2-13}).
		This completes the proof.
	\end{proof}

	Let $h_{X,0}$ be an $\iota$-invariant Ricci-flat \K metric on $X$
	with \K form $\omega_{X,0}$.
	It is locally written as
	$$
		\omega_{X,0} = \frac{i}{2} \sum_{j,k} h_{X,0} \left( \frac{\partial}{\partial z^j}, \frac{\partial}{\partial z^k} \right) dz^j \wedge d\bar{z}^k .
	$$
	The Riemannian metric associated with $h_{X,0}$ is denoted by $g$.
	The hermitian metric on $\wedge^{p,q} T^*X$ attached to the Ricci-flat \K metric is also donoted by $h_{X,0}$
	
	Let $I$ be the complex structure of $X$.
	Since $(X,g)$ is hyperk\"ahler, there are complex structures $J$ and $K$ of $X$ such that $(X, I, g)$, $(X, J, g)$ and $(X, K, g)$ are manifolds of $K3^{[2]}$-type and 
	$
		IJ=-JI=K.
	$
	The \K forms with respect to $J$ and $K$ are given by $\omega_J =g(-, J(-))$ and $\omega_K =g(-, K(-))$, respectively.
	Set $\sigma_I = \omega_J +i \omega_K$.
	This is a holomorphic 2-form on $X$ (cf. \cite[\S 23]{MR1963559}).
	Since $H^0(X, \Omega_X^2) =\mathbb{C} \eta$, there exists a complex number $\lambda \in \mathbb{C}$ such that $\eta = \frac{\lambda}{2} \sigma_I$.
	Note that the $L^2$-norm of $\eta^2$ is given by
	$
		\| \eta^2 \|_{L^2}^2  = \int_X \eta^2 \wedge \bar{\eta}^2 ,
	$
	and the volume of $(X, \omega_{X,0})$ is defined by
	$
		\vol(X, \omega_{X,0}) = \int_X \omega_{X,0}^4/4!.
	$
	Since $\omega_{X,0}$ is Ricci-flat, it follows from \cite[Corollary 23.9]{MR1963559} that
	\begin{align}\label{f-3-1}
		|\lambda|^2 = \frac{1}{2} \left( \frac{\eta^2 \wedge \bar{\eta}^2}{\omega_{X,0}^4/4!} \right)^{\frac{1}{2}} =\frac{1}{2} \left( \frac{\| \eta^2 \|_{L^2}^2}{\vol(X, \omega_{X,0})} \right)^{\frac{1}{2}}.
	\end{align}
	Similarly, there exists a complex number $\mu \in \mathbb{C}$ such that $\theta = \frac{\mu}{2} \bar{\sigma}_I$ in $H^2(X, \mathcal{O}_X)$.
	We identify the cohomology class $\theta$ with its harmonic representative.
	 The $L^2$-norm of $\theta^2$ is given by
	$
		\| \theta^2 \|_{L^2}^2  = \int_X \theta^2 \wedge \bar{\theta}^2 ,
	$
	and we have
	\begin{align}\label{f-3-2}
		|\mu|^2  =\frac{1}{2} \left( \frac{ \| \theta^2 \|_{L^2}^2}{\vol(X, \omega_{X,0})} \right)^{\frac{1}{2}}.
	\end{align}
	
	\begin{lem}\label{l-3-1}
		For any $\alpha \in A^{1,1}(X)$, the following identity holds:
		$$
			h_{X,0}(\eta \wedge \alpha, \eta \wedge \alpha) = |\lambda|^2 h _{X,0}(\alpha, \alpha).
		$$
	\end{lem}
	
	\begin{proof} 
		Fix $p \in X$.
		Since $(X, g)$ is hyperk\"ahler, the real tangent space $T_{p, \mathbb{R}}X$ at $p$ is equipped with the structure of a quaternionic hermitian vector space
		(see \cite[\S 23.2]{MR1963559}).
		Therefore, there exist $e, f \in T_{p, \mathbb{R}}X$ such that 
		$
			\{ e, Ie, Je, Ke, f, If, Jf, Kf \}
		$
		forms an orthonormal basis of $(T_{p, \mathbb{R}}X, g)$.
		Set 
		$$
			v_1 =\frac{ e-iIe }{2}, \quad v_2 =\frac{ -Ke-iJe }{2}, \quad v_3 =\frac{ f-iIf }{2}, \quad v_4 =\frac{ -Kf-iJf }{2}.
		$$
		Then $v_1, v_2, v_3, v_4$ are of type $(1,0)$ with respect to the complex structure $I$
		and form a $\mathbb{C}$-basis of the holomorphic tangent space $T_pX$ at $p$.
		Moreover, we have
		$
			h_{X,0}(v_i, v_j) = \frac{1}{2} \delta_{ij}.
		$
		Let $v^1, v^2, v^3, v^4$ be the dual basis. 
		Then we have $h_{X,0}(v^i, v^j) = 2\delta_{ij}$ and $\sigma_I=i v^1 \wedge v^2 +i v^3 \wedge v^4$.
		%$$
		%	h_{X,0}(v^i, v^j) = 2\delta_{ij}, \quad \text{and} \quad \sigma_I=i v^1 \wedge v^2 +i v^3 \wedge v^4.
		%$$
		Let $\alpha = \frac{1}{2} \sum_{i,j} \alpha_{ij} v^i \wedge \bar{v}^j \in \wedge^{1,1}T_p^*X$.
		Then we have
		$$
			h_{X,0}(\sigma_I \wedge \alpha, \sigma_I \wedge \alpha) = 4\sum_{i,j} |\alpha_{ij}|^2 = 4 h _{X,0}(\alpha, \alpha),
		$$
		which completes the proof.
	\end{proof}

	\begin{lem}\label{l-2-2-14}
		The following identities hold:
		\begin{align*}
			& \quad c_1 \left( \left( R^1f_*\Omega^1_{\xs} \right)_- , h_{L^2} \right) -c_1 \left( \left( R^3f_*\Omega^1_{\xs} \right)_+ , h_{L^2} \right) \\
			&=-\frac{21 - t}{4} c_1(R^4f_*\mathcal{O}_{\mathscr{X}}, h_{L^2}) -\frac{21 - t}{4} dd^c \log \vol (\xs, \omega) ,\\
			& \quad c_1 \left( \left( R^1f_*\Omega^1_{\xs} \right)_- , h_{L^2} \right) -c_1 \left( \left( R^1f_*\Omega^3_{\xs} \right)_+ , h_{L^2} \right) \\
			&=-\frac{21 - t}{4} c_1(f_*\Omega^4_{\xs}, h_{L^2} ) -\frac{21 - t}{4} dd^c \log \vol (\xs, \omega).
		\end{align*}
	\end{lem}
	
	\begin{proof}
		The first identity is proved in \cite[Lemma 3.7]{I1}.
		We prove the second identity in the same manner as in \cite[Lemma 3.7]{I1}.
		
		Let $X$ be a fiber of $f : \X \to S$.
		Since $X$ is a manifold of $K3^{[2]}$-type, it follows from \cite[Proposition 24.1]{MR1963559} and \cite[Main Theorem]{MR1879810} that the cup product induces a $\G$-equivariant isomorphism
		$
			\operatorname{Sym}^2 H^2(X, \mathbb{C}) \cong H^4(X, \mathbb{C}).
		$
		Since the involution $\iota : X \to X$ is antisymplectic, we have the canonical isomorphism
		\begin{align*}
			 H^1(X, \Omega^1_X)_{-} \otimes H^{0}(X, \Omega^{2}_{X}) \cong H^1(X, \Omega^3_X)_{+} \quad \alpha \otimes \eta \to \alpha \wedge \eta.
		\end{align*}
		Therefore, we have an isomorphism of holomorphic vector bundles
		\begin{align}\label{al3-2-2-3}
			\left( R^1f_*\Omega^1_{\xs} \right)_{-} \otimes f_*\Omega^{2}_{\xs} \cong \left( R^1f_*\Omega^3_{\xs} \right)_+.
		\end{align}
		
		Set $N= \frac{21-t}{2}$.
		Let $s_1, \dots , s_N$ be a local holomorphic frame of $\left( R^1f_*\Omega^1_{\xs} \right)_-$.
		Then $s_1 \wedge \dots \wedge s_N$ is a nowhere vanishing local holomorphic section of $\det \left( R^1f_*\Omega^1_{\xs} \right)_-$.
		Let $\eta$ be a nowhere vanishing local holomorphic section of $f_*\Omega^2_{\xs}$.
		By (\ref{al3-2-2-3}), $s_1 \wedge \eta, \dots , s_N \wedge \eta$ are local holomorphic frame of $\left( R^1f_*\Omega^3_{\xs} \right)_+$
		and $(s_1 \wedge \eta) \wedge \dots \wedge (s_N \wedge \eta)$ is a nowhere vanishing local holomorphic section of $\det \left( R^1f_*\Omega^3_{\xs} \right)_+$.
		
		It suffices to show that
		\begin{align}\label{f-3-7}
		\begin{aligned}
			&-\log \| (s_1 \wedge \eta) \wedge \dots \wedge (s_N \wedge \eta) \|_{L^2}^2 \\
			&=-\frac{N}{2} \log \| \eta^2 \|_{L^2}^2 -\log \| s_1 \wedge \dots \wedge s_N \|_{L^2}^2 +\log \left( 2^N \vol (\xs, \omega_{\xs})^{\frac{N}{2}} \right).
		\end{aligned}
		\end{align}
		We may assume that $S$ consists of one point,
		and we set $X=\mathscr{X}$.\par
		By Lemma \ref{l-1-4-4}, the $L^2$-metrics on $H^1(X, \Omega_X^p)$ $(p=1,3)$ and $H^0(X, \Omega^2_X)$ depend only on the choice of the \K class $[\omega_X]$ and are independent of the \K form $\omega_X$ itself.
		Therefore we may assume that $\omega_X =\omega_{X,0}$ is Ricci-flat.\par
		By Lemma \ref{l-3-1}, we have 
		$
			h_{X,0}(\eta \wedge \alpha, \eta \wedge \alpha) = |\lambda|^2 h _{X,0}(\alpha, \alpha)
		$
		for each $\alpha \in A^{1,1}(X)$.
		By integrating both sides, we have
		$
			\| \eta \wedge \alpha \|_{L^2}^2 =|\lambda|^2 \| \alpha \|_{L^2}^2.
		$
		Therefore we have an isometry
		$$
			(\det H^1(X, \Omega_X^1)_- , |\lambda|^{2N} \det h_{L^2}) \cong (\det H^1(X, \Omega_X^3)_+, \det h_{L^2}).
		$$
		By the formula (\ref{f-3-1}), we obtain the formula (\ref{f-3-7}).
	\end{proof}

	\begin{lem}\label{l-2-2-15}
		The following identity holds:
		\begin{align*}
			c_1( R^4h_* \mathcal{O}_{\Z}, h_{L^2} ) = - c_1( f_*\Omega^4_{\xs}, h_{L^2} ).
		\end{align*}
	\end{lem}
	
	\begin{proof}
		Recall that $\tau : \widetilde{\X} \to \X$ is the blowup of $\X$ along $\XX$.
		Let $\pi : \widetilde{\X} \to \Z = \widetilde{\X} / \tilde{\iota}$ be the natural projection.
		Set $\tilde{f} = f \circ \tau : \widetilde{\X} \to S$.
		For $s \in S$, the fiber of $f$, $\tilde{f}$ and $h$ are denoted by $X_s$, $\tilde{X}_s$ and $Z_s$, respectively. 
		Since $h^{0,4}(X_s) = h^{0,4}(\tilde{X}_s) = h^{0,4}(Z_s)=1$, $R^4f_* \mathcal{O}_{\X}$, $R^4 \tilde{f}_* \mathcal{O}_{\widetilde{\X}}$ and $R^4h_* \mathcal{O}_{\Z}$ are invertible sheaves.
		
		Let $\alpha \in H^4( X_s, \mathcal{O}_{X_s})$.
		Since $\iota_s$ is an antisymplectic involution on $X_s$, $\alpha$ is $\iota_s$-invariant
		and $\tau^* \alpha \in H^4( \tilde{X}_s, \mathcal{O}_{ \tilde{X}_s })$ is $\tilde{\iota}_s$-invariant.
		Therefore, there exists a unique element $\beta \in H^4( Z_s, \mathcal{O}_{Z_s})$ such that
		\begin{align}\label{al-2-2-16}
			\tau^* \alpha = \pi^* \beta.
		\end{align}
		The correspondence from $\alpha$ to $\beta$ induces a natural isomorphism of holomorphic line bundles on $S$
		\begin{align}\label{al-2-2-17}
			R^4f_*\mathcal{O}_{\X} \cong R^4h_*\mathcal{O}_{\Z} .
		\end{align}
		Since $\pi : \widetilde{\X} \to \Z$ is a double cover, we have
		\begin{align}\label{al-2-2-18}
			\| \beta \|^2_{L^2} &= \int_{Z_s} \beta \wedge \bar{\beta} 
			= \frac{1}{2} \int_{\tilde{X}_s} \pi^* \beta \wedge \overline{\pi^* \beta} 
			= \frac{1}{2} \int_{X_s} \alpha \wedge \bar{\alpha} 
			= \frac{1}{2} \| \alpha \|^2_{L^2},
		\end{align}
		where the third equality holds by $(\ref{al-2-2-16})$.
		By (\ref{al-2-2-17}), (\ref{al-2-2-18}) and the Serre duality, we have
		$
			c_1( R^4h_* \mathcal{O}_{\Z}, h_{L^2} ) =  c_1( R^4f_*\mathcal{O}_{\X}, h_{L^2} ) = - c_1( f_*\Omega^4_{\xs}, h_{L^2} ),
		$
		which completes the proof.
	\end{proof}

	\begin{lem}\label{l-2-2-19}
		The following identity holds:
		\begin{align*}
			\cpq{1}{3} = -\frac{21-t}{4} c_1( f_*\Omega^4_{\xs}, h_{L^2} ) -c_1( f_*\Omega^2_{\xss}, h_{L^2} ).
		\end{align*}
	\end{lem}
	
	\begin{proof}
		By the Serre duality and Lemma \ref{l-2-2-5}, we have
		\begin{align}\label{al2-2-2-8}
		\begin{aligned}
			&\quad \cpq{1}{3} \\
			&= \frac{1}{2} \left\{ \cpq{1}{3} -\cpq{3}{1} \right\} \\
			&= \frac{1}{2} \left\{ c_1 \left( \left( R^3f_*\Omega^1_{\xs} \right)_{+}, h_{L^2} \right) -c_1 \left(  \left( R^1f_*\Omega^3_{\xs} \right)_{+}, h_{L^2} \right) \right\} \\
			&\quad +\frac{1}{2} \left\{ c_1( R^2f_* \mathcal{O}_{\XX}, h_{L^2} ) -c_1( f_*\Omega^2_{ \xss }, h_{L^2} ) \right\}.
		\end{aligned}
		\end{align}
		By Lemma \ref{l-2-2-14} and the Serre duality, we have
		\begin{align}\label{al-2-2-23}
			c_1 \left( \left( R^3f_*\Omega^1_{\xs} \right)_{+}, h_{L^2} \right) -c_1 \left(  \left( R^1f_*\Omega^3_{\xs} \right)_{+}, h_{L^2} \right) = -\frac{21-t}{2} c_1( f_*\Omega^4_{\xs}, h_{L^2} ).
		\end{align}
		By the Serre duality, we have
		\begin{align}\label{al-2-2-24}
			c_1( R^2f_* \mathcal{O}_{\XX}, h_{L^2} ) =-c_1( f_*\Omega^2_{ \xss }, h_{L^2} ).
		\end{align}
		Substituting (\ref{al-2-2-23}) and (\ref{al-2-2-24}) into (\ref{al2-2-2-8}), we obtain the desired formula.
	\end{proof}

	\begin{lem}\label{l-2-2-26}
		The following identity holds:
		\begin{align*}
			2 c_1( R^2h_*\mathbb{C} \otimes \mathcal{O}_S, h_{L^2} ) - c_1( R^3h_*\mathbb{C} \otimes \mathcal{O}_S, h_{L^2} ) = -4 c_1( h_*\mathcal{O}_{\Z}, h_{L^2} ).
		\end{align*}
	\end{lem}
	
	\begin{proof}
		By \cite[proof of Proposition 5.9]{MR4255041}, we have
		\begin{align}\label{al-2-2-27}
			\sum^8_{k=0} (-1)^k k c_1( R^k h_*\mathbb{C} \otimes \mathcal{O}_S, h_{L^2} ) =0.
		\end{align}
		By (\ref{al-2-2-2}), we have
		\begin{align}\label{al-2-2-28}
			c_1( R^1h_*\mathbb{C} \otimes \mathcal{O}_S, h_{L^2} ) = c_1( R^7h_*\mathbb{C} \otimes \mathcal{O}_S, h_{L^2} ) =0.
		\end{align}
		By the Poincar\'e duality, we have
		\begin{align}\label{al-2-2-29}
		\begin{aligned}
			c_1( R^4h_*\mathbb{C} \otimes \mathcal{O}_S, h_{L^2} ) &= 0, \\
			c_1( R^5h_*\mathbb{C} \otimes \mathcal{O}_S, h_{L^2} ) &= -c_1( R^3h_*\mathbb{C} \otimes \mathcal{O}_S, h_{L^2} ), \\
			c_1( R^6h_*\mathbb{C} \otimes \mathcal{O}_S, h_{L^2} ) &= -c_1( R^2h_*\mathbb{C} \otimes \mathcal{O}_S, h_{L^2} ), \\
			c_1( R^8h_*\mathbb{C} \otimes \mathcal{O}_S, h_{L^2} ) &= -c_1( h_*\mathbb{C} \otimes \mathcal{O}_S, h_{L^2} ) = - c_1( h_*\mathcal{O}_{\Z}, h_{L^2} ).
		\end{aligned}
		\end{align}
		Substituting (\ref{al-2-2-28}), (\ref{al-2-2-29}) into (\ref{al-2-2-27}), we obtain the desired formula.
	\end{proof}

	Recall that $\omega_{ H^{\cdot}( \xss )}$ is defined by
	\begin{align}\label{al-2-2-30}
		\omega_{H^{\cdot}(\xss)} = c_1(f_*\Omega^1_{\xss}, h_{L^2}) -c_1(R^1f_*\mathcal{O}_{\mathscr{X}^{\iota}}, h_{L^2}) -2c_1(f_*K_{\xss}, h_{L^2}).
	\end{align}
	
	\begin{prop}\label{p-2-2-31}
		The following identity holds:
		\begin{align*}
			\sum^8_{k=0} (-1)^k \omega_{H^k} = \frac{29-t}{2} c_1( f_*\Omega^4_{\xs}, h_{L^2} ) -\omega_{ H^{\cdot}( \xss )}.
		\end{align*}
	\end{prop}
	
	\begin{proof}
		By Lemmas \ref{l-2-2-11} and \ref{l-2-2-19}, we have
		\begin{align}\label{al-2-2-32}
		\begin{aligned}
			&\quad 2\sum^3_{q=1} (-1)^q c_1( R^qh_* \Omega^1_{\zs}, h_{L^2} )  \\
			&= \frac{21-t}{2} c_1( f_*\Omega^4_{\xs}, h_{L^2} ) \\
			&\quad -\left\{ c_1(f_*\Omega^1_{\xss}, h_{L^2}) -c_1(R^1f_*\mathcal{O}_{\mathscr{X}^{\iota}}, h_{L^2}) -2c_1(f_*K_{\xss}, h_{L^2}) \right\} \\
			&\quad -\left\{ 2 c_1( R^2h_*\mathbb{C} \otimes \mathcal{O}_S, h_{L^2} ) - c_1( R^3h_*\mathbb{C} \otimes \mathcal{O}_S, h_{L^2} ) \right\}.
		\end{aligned}
		\end{align}
		Substituting (\ref{al-2-2-30}) and Lemma \ref{l-2-2-26} into (\ref{al-2-2-32}), we have
		\begin{align}\label{al-2-2-33}
			2\sum^3_{q=1} (-1)^q c_1( R^qh_* \Omega^1_{\zs}, h_{L^2} ) = \frac{21-t}{2} c_1( f_*\Omega^4_{\xs}, h_{L^2} ) -\omega_{H^{\cdot}(\xss)} +4 c_1( h_*\mathcal{O}_{\Z}, h_{L^2} ).
		\end{align}
		By (\ref{al-2-2-33}) and by Lemmas \ref{l-2-2-1} and \ref{l-2-2-15}, we have
		$$
			\sum^8_{k=0} (-1)^k \omega_{H^k} = \frac{29-t}{2}c_1( f_*\Omega^4_{\xs}, h_{L^2} ) -\omega_{H^{\cdot}(\xss)},
		$$
		which completes the proof.
	\end{proof}

	For a compact \K manifold $M$, we set
	$
		h^{p,q}(M) = \dim H^q(M, \Omega^{p}_M), b_k(M) = \dim H^k(M, \mathbb{C}), \chi(M) = \sum_{k \geqq 0} (-1)^k b_k(M).
	$
	Let $X$ and $Z$ be the fibers of $f : \X \to S$ and $h : \Z \to S$, respectively.
	Let $H^q(X, \Omega^p_X)_{\pm}$ be the $(\pm 1)$-eigenspaces of $H^q(X, \Omega^p_X)$ with respect to the $\G$-action induced from $\iota : X \to X$.
	Set $h^{p,q}(X)_{\pm} = \dim H^q(X, \Omega^p_X)_{\pm}$ and $b_k(X)_{\pm} = \sum_{p+q=k} h^{p,q}(X)_{\pm}$.
	%$$
	%	h^{p,q}(X)_{\pm} = \dim H^q(X, \Omega^p_X)_{\pm}, \quad b_k(X)_{\pm} = \sum_{p+q=k} h^{p,q}(X)_{\pm}.
	%$$
	Recall that $t = \operatorname{Tr} (\iota^*|_{H^{1,1}(X)})$.
	Therefore
	$$
		h^{1,1}(X)_+ -h^{1,1}(X)_- =t, \quad h^{1,1}(X)_+ +h^{1,1}(X)_- = h^{1,1}(X) =21,
	$$
	and 
	\begin{align}\label{al-2-2-35}
		h^{1,1}(X)_+ = \frac{21+t}{2} , \quad h^{1,1}(X)_- = \frac{21-t}{2}.
	\end{align}
	By \cite[proof of Theorem 4.1]{MR3928256}, we have
	$
		H^q(X, \Omega^p_X)_{+} \oplus H^{q-1} (X^{\iota}, \Omega^{p-1}_{X^{\iota}}) \cong H^q(Z, \Omega^p_Z).
	$
	Since
	\begin{align*}
		b_k(Z) = \sum_{p+q=k} h^{p,q}(Z) = \sum_{p+q=k} \left( h^{p,q}(X)_+ +h^{p-1,q-1}(X^{\iota}) \right) =b_k(X)_+ +b_{k-2}(X^{\iota}),
	\end{align*}
	we have
	\begin{align}\label{al-2-2-36}
		\chi(Z) = \sum^8_{k=0} (-1)^k b_k(X)_+ +\chi(X^{\iota}). 
	\end{align}
	
	\begin{prop}\label{p-2-2-37}
		The following identity holds:
		\begin{align*}
			\chi(Z) = \frac{3 t^2 + 717}{4}.
		\end{align*}
	\end{prop}
	
	\begin{proof}
		By \cite[Theorem 2]{MR2805992}, we have
		\begin{align}\label{al-2-2-38}
			\chi(X^{\iota}) = \frac{t^2 + 23}{2}. 
		\end{align}
		By \cite[Main Theorem]{MR1879810}, we have $b_k(X) =0$ for $k= 1, 3, 5, 7$, and
		\begin{align}\label{al-2-2-39}
			b_k(X)_+ = 0 \quad (k= 1, 3, 5, 7).
		\end{align}
		Recall that $\iota$ is antisymplectic. 
		By (\ref{al-2-2-35}), we have
		\begin{align}\label{al-2-2-40}
		\begin{aligned}
			b_0(X)_+ = b_8(X)_+ =1 ,\quad
			b_2(X)_+ = h^{1,1}(X)_+ = \frac{21+t}{2} = b_6(X)_+.
		\end{aligned}
		\end{align}
		On the other hand,
		\begin{align}\label{al-2-2-41}
			b_2(X)_- = h^{2,0}(X) +h^{1,1}(X)_- +h^{0,2}(X) =\frac{25-t}{2}.
		\end{align}
		By \cite[Proposition 24.1]{MR1963559}, the cup product induces $\operatorname{Sym}^2 H^2(X, \mathbb{C}) \cong H^4(X, \mathbb{C})$.
		%$
		%	\operatorname{Sym}^2 H^2(X, \mathbb{C}) \cong H^4(X, \mathbb{C}).
		%$
		Therefore
		$
			H^4(X, \mathbb{C})_+ \cong \operatorname{Sym}^2 H^2(X, \mathbb{C})_+ \oplus \operatorname{Sym}^2 H^2(X, \mathbb{C})_-,
		$
		and, by (\ref{al-2-2-40}), (\ref{al-2-2-41}), we have
		\begin{align}\label{al-2-2-42}
		\begin{aligned}
			b_4(X)_+ = \dbinom{b_2(X)_+}{2} +b_2(X)_+ +\dbinom{b_2(X)_-}{2} +b_2(X)_-
			= \frac{t^2 -4t +579}{4}.
		\end{aligned}
		\end{align}
		By (\ref{al-2-2-39}), (\ref{al-2-2-40}) and (\ref{al-2-2-42}), we have
		\begin{align}\label{al-2-2-43}
			\sum^8_{k=0} (-1)^k b_k(X)_+ = \frac{t^2 + 671}{4}.
		\end{align}
		Substituting (\ref{al-2-2-38}), (\ref{al-2-2-43}) into (\ref{al-2-2-36}), we have the desired formula.
		%$$
		%	\chi(Z) = \frac{3 t^2 + 717}{4}.
		%$$
		%This completes the proof.
	\end{proof}

	\textbf{Proof of Theorem \ref{t-2-1-3}}.
	By Lemma \ref{l-2-2-15} and by Propositions \ref{p-2-2-31}, \ref{p-2-2-37}, we have
	\begin{align}\label{al-2-2-44}
	\begin{aligned}
		&\quad \sum^{8}_{k=0} (-1)^k \omega_{H^k} -\frac{\chi}{12} \omega_{WP} \\
		&=  \frac{29-t}{2} c_1( f_*\Omega^4_{\xs}, h_{L^2} ) -\omega_{ H^{\cdot}( \xss )} -\frac{ t^2 + 239}{16} c_1(h_*K_{\zs}, h_{L^2}) \\
		&=  -\frac{(t+1)(t+7)}{16} c_1( f_*\Omega^4_{\xs}, h_{L^2} ) -\omega_{ H^{\cdot}( \xss )}.
	\end{aligned}
	\end{align}
	By (\ref{al-2-2-44}) and by Theorem \ref{p-3-4}, we have
	\begin{align*}
		dd^c \log \tau_{M, \mathcal{K}, \xs} =  -\frac{(t+1)(t+7)}{16} c_1( f_*\Omega^4_{\xs}, h_{L^2} ) -\omega_{ H^{\cdot}( \xss )} = \sum^{8}_{k=0} (-1)^k \omega_{H^k} -\frac{\chi}{12} \omega_{WP}.
	\end{align*}
	This is the desired result. \hfill $\Box$

\subsection{The BCOV invariant of the Calabi-Yau fourfolds $Z$} \label{ss-2-3}
	
	Let $M_0$ be a primitive hyperbolic 2-elementary sublattice of $L_{K3}$. 
	As before, the orthogonal complement of $M_0$ is considered in $L_{K3}$.
	In this subsection, we assume that 
	\begin{align*}
		(1) \hspace{3pt} \rk (M_0) \leqq 17, \qquad (2)\hspace{3pt} \text{$\bar{\mathscr{D}}_{M_0^{\perp}}$ is irreducible}.
	\end{align*}
	%\begin{enumerate}[ label= \rm{(\arabic*)} ]
	%	\item $\rk (M_0) \leqq 17$.
	%	\item $\bar{\mathscr{D}}_{M_0^{\perp}}$ is irreducible.
	%\end{enumerate}
	Here $M_0^{\perp}$ denotes the orthogonal complement of $M_0$ in $L_{K3}$.
	We define a sublattice $\widetilde{M}_0$ of $L_2 = L_{K3} \oplus \mathbb{Z}e$ by
	$$
		\widetilde{M}_0 = M_0  \oplus \mathbb{Z}e.
	$$
	Then $\widetilde{M}_0$ is an admissible sublattice of $L_2$.
	Let $\mathcal{K} \subset \operatorname{KT}(\widetilde{M}_0)$ be a natural chamber.
	
	Let $(Y, \sigma)$ be a 2-elementary K3 surface of type $M_0$.
	Then $(Y^{[2]}, \sigma^{[2]})$ is a manifold of $K3^{[2]}$-type with involution of type $(\widetilde{M}_0, \mathcal{K})$.
	We set $(X_Y, \iota_Y)=(Y^{[2]}, \sigma^{[2]})$.
	The blowup of $X_Y$ along the fixed locus $X_Y^{\iota_Y}$ is denoted by $\widetilde{X_Y}$,
	and the involution of $\widetilde{X_Y}$ induced from $\iota_Y$ is denoted by $\tilde{\iota}_Y$.
	By \cite[Theorem 3.6]{MR3928256}, $\widetilde{X_Y} / \tilde{\iota}_Y$ is a Calabi-Yau 4-fold.
	Set
	$$
		Z_Y = \widetilde{X_Y} / \tilde{\iota}_Y.
	$$
	
	Recall that $\mathcal{M}_{M_0}^{\circ} = \mathcal{M}_{M_0} \setminus \bar{\mathscr{D}}_{M_0^{\perp}}$ is a coarse moduli space of 2-elementary K3 surfaces of type $M_0$.
	We define a function $u$ on $\mathcal{M}_{M_0}^{\circ}$ by
	$$
		u(Y, \sigma) = \tau_{BCOV} (Z_Y) \cdot \tau_{\widetilde{M}_0, \mathcal{K}} (X_Y, \iota_Y)^{-1} \quad ( (Y, \sigma) \in \mathcal{M}_{M_0}^{\circ} ).
	$$
	
	\begin{lem}\label{l-2-3-1}
		The function $\log u$ on $\mathcal{M}_{M_0}^{\circ}$ is pluriharmonic.
	\end{lem}
	
	\begin{proof}
		Let $\Pi_{M_0^{\perp}} : \Omega_{M_0^{\perp}} \to \mathcal{M}_{M_0}$ be the natural projection.
		It suffices to show that $\Pi_{M_0^{\perp}}^*u$ is a smooth function on $ \Omega_{M_0^{\perp}} \setminus \mathscr{D}_{M_0^{\perp}}$
		and $\log \Pi_{M_0^{\perp}}^*u$ is a pluriharmonic function on $ \Omega_{M_0^{\perp}} \setminus \mathscr{D}_{M_0^{\perp}}$.
		
		Let $x \in \Omega_{M_0^{\perp}} \setminus \mathscr{D}_{M_0^{\perp}}$ and set $p = \Pi_{M_0^{\perp}} (x)$.
		Let $(Y, \sigma) \in \bar{\pi}_{M_0}^{-1}(p)$.
		By the local Torelli theorem, we may assume that the Kuranishi space $S := \operatorname{Def}(Y, \sigma)$ is an open neighborhood of $x$ in $ \Omega_{M_0^{\perp}} $.
		Its Kuranishi family is denoted by $g : (\mathscr{Y}, \sigma) \to S$.
		Let $f : (\X, \iota) \to S$ be the family of $K3^{[2]}$-type manifolds with involution of type $(\widetilde{M}_0, \mathcal{K})$ such that
		$
			(X_s, \iota_s) = (X_{Y_s}, \iota_{Y_s})
		$
		for each $s \in S$.
		Let $h : \Z \to S$ be the family of Calabi-Yau 4-folds constructed in the same manner as in \S 2.1.
		
		Then $\left( \Pi_{M_0^{\perp}}^*u \right)|_S = \tau_{BCOV, \zs} \cdot \tau_{\widetilde{M}_0, \mathcal{K}, \xs}^{-1}$.
		Since both $\tau_{BCOV, \zs}$ and $\tau_{\widetilde{M}_0, \mathcal{K}, \xs}$ are smooth functions on $S$, so is $\left( \Pi_{M_0^{\perp}}^*u \right)|_S$.
		By (\ref{al-2-1-1}) and Theorem \ref{t-2-1-3}, $\log \left( \Pi_{M_0^{\perp}}^*u \right)|_S \\ \left( = \log \tau_{BCOV, \zs} -\log \tau_{\widetilde{M}_0, \mathcal{K}, \xs} \right)$ is a pluriharmonic function on $S$. %無理矢理改行
	\end{proof}

	Let $C$ be an irreducible projective curve on $\mathcal{M}_{M_0}^*$ such that $C$ is not contained in $\bar{\mathscr{D}}_{M_0^{\perp}} \cup (\mathcal{M}^*_{M_0} \setminus \mathcal{M}_{M_0})$.
	%Namely, $C$ is not contained in $\bar{\mathscr{D}}_{M^{\perp}} \cup (\mathcal{M}^*_{M_0} \setminus \mathcal{M}_{M_0})$.
	Let $p \in C \cap \bar{\mathscr{D}}_{M_0^{\perp}}$ and assume that $p$ is a smooth point of $C$.
	By \cite[Proposition 2.14]{I2}, there exists a smooth projective curve $B$,
	a morphism $\varphi: B \to \mathcal{M}^*_{M_0}$,
	a smooth projective variety $\mathscr{X}$ of complex dimension $5$,
	an involution $\iota : \mathscr{X} \to \mathscr{X}$,
	and a surjective morphism $f : \mathscr{X} \to B$ such that
	\begin{enumerate}[ label= \rm{(\arabic*)} ]
		\item $\varphi(B)=C$.
		\item $f \circ \iota =f$.
		\item There exists a Zariski open subset $B^{\circ} \neq \emptyset$ of $B$ such that $\varphi(B^{\circ}) \subset \mathcal{M}^{\circ}_{M_0}$
			and for each $b \in B^{\circ}$, the fiber $(f^{-1}(b), \iota|_{f^{-1}(b)})$ is a manifold of $K3^{[2]}$-type with antisymplectic involution of type $(\widetilde{M}_0, \mathcal{K})$ with period $\varphi(b)$. 
		\item $f : \mathscr{X} \to B$ is a normal crossing degeneration. 
			Namely, its singular fibers are simple normal crossing divisor on $\X$. 
	\end{enumerate}
	
	Let $\Sigma_f = \{ s \in \X ; df(x) =0 \}$ be the critical locus of $f$
	and let $\Delta_f = f (\Sigma_f)$ be the discriminant locus of $f$.
	The fixed locus of $\iota$ is denoted by $\X^{\iota}$.
	Write
	$$
		\X^{\iota} = \X^{\iota}_H \sqcup \X^{\iota}_V,
	$$ 
	where any connected component of $\X^{\iota}_H$ is flat over $B$ and $\X^{\iota}_V \subset f^{-1}(\Delta_f)$.
	
	The blowup of $\X$ along $\X^{\iota}_H$ is denoted by $\widetilde{\X}$.
	Let $\tilde{\iota}$ be the involution of $\widetilde{\X}$ induced from $\iota$.
	Set $\Z = \widetilde{\X} / \tilde{\iota}$ and let $h : \Z \to B$ be the morphism induced from $\widetilde{\X} \to \X \xrightarrow{f} B$.
	By \cite[Theorem 3.6]{MR3928256} and $(3)$, $h|_{ h^{-1}(B^{\circ}) } : h^{-1}(B^{\circ}) \to B^{\circ}$ is a family of Calabi-Yau 4-folds.
	
	Let $0 \in B \setminus B^{\circ}$ and let $(\D', t)$ be a local coordinate on $B$ centered at $0$ such that $\D' \setminus \{ 0 \} \subset B^{\circ}$.
	We define functions $\tau_{BCOV, \Z /B}$ and $\tau_{\widetilde{M}_0, \mathcal{K}, \X /B}$ on $B^{ \circ }$ by
	$$
		\tau_{BCOV, \Z /B}(b) = \tau_{BCOV}( Z_b ), \quad \tau_{\widetilde{M}_0, \mathcal{K}, \X /B}(b) = \tau_{\widetilde{M}_0, \mathcal{K}}( X_b, \iota_b) \quad (b \in B^{\circ}).
	$$
	Then $ \varphi^*u = \tau_{BCOV, \Z /B} \cdot \tau_{\widetilde{M}_0, \mathcal{K}, \X /B}^{-1}$.
	By \cite[Theorem 6.5 and Proposition 6.8]{MR4255041} and \cite[Theorem 2.12]{I2}, 
	there exists $a_0 \in \mathbb{Q}$ such that
	\begin{align}\label{al2-2-3-1}
		\log \varphi^*u|_{\D'} (t) = a_0 \log |t|^2 +O( \log \log |t|^{-1}) \quad (t \to 0).
	\end {align}
	
	Let $(\D, s)$ be a local coordinate on $C$ centered at $p$ such that $\D^* \subset \mathcal{M}^{\circ}_{M_0}$.
	Since $\varphi : B \to C$ is a finite map with $\varphi(0) =p$, 
	there exists $\nu \in \mathbb{Z}_{>0}$ such that $\varphi^* s = t^{\nu} \varepsilon(t)$,
	where $\varepsilon(t) \in \mathbb{C} \{ t \}$ is a unit.
	Replacing $t$ with $t \varepsilon(t)^{\frac{1}{\nu}}$, we may assume $\varepsilon(t)=1$.
	Namely, $\varphi^* s = t^{\nu}$.
	Set $a =a_0 / \nu$.
	By (\ref{al2-2-3-1}), we have
	\begin{align}\label{al-2-3-2}
		\log u|_{\D} (s) = a \log |s|^2 +O( \log \log |s|^{-1}) \quad (s \to 0).
	\end {align}
	
	Since $\bar{\mathscr{D}}_{M_0^{\perp}}$ is irreducible, it follows from (\ref{al-2-3-2}) and  \cite[Lemma 4.2]{I2} the equation of currents on $\mathcal{M}_{M, \mathcal{K}}$
	\begin{align}\label{al-2-3-3}
		dd^c \log u = a \delta_{\bar{\mathscr{D}}_{M_0^{\perp}}} ,
	\end {align}
	where $ \delta_{\bar{\mathscr{D}}_{M_0^{\perp}}} $ is the Dirac delta current with support $\bar{\mathscr{D}}_{M_0^{\perp}}$.
	
	\begin{thm}\label{t-2-3-4}
		If $r(M_0) \leqq 17$ and $\bar{\mathscr{D}}_{M_0^{\perp}}$ is irreducible, 
		then there exists a constant $C_{M_0}$ depending only on $M_0$ such that for any 2-elementary K3 surface $(Y, \sigma)$ of type $M_0$
		$$
			\tau_{BCOV}(Z_Y) =C_{M_0} \tau_{\widetilde{M}_0, \mathcal{K}} (X_Y, \iota_Y).
		$$ 
	\end{thm}
	
	\begin{proof}
		We show that $\log u$ is a constant function on $\mathcal{M}_{M_0}^{\circ}$.
		Since $r(M_0) \leqq 17$, we have $ \operatorname{codim} \mathcal{M}_{M_0}^* \setminus \mathcal{M}_{M_0} \geqq 2$.
		Therefore, we can choose an irreducible projective curve $C$ on $\mathcal{M}_{M_0}^*$ that satisfies the following conditions:
		\begin{align*}
			&(1) \hspace{5pt} C \subset \mathcal{M}_{M_0} \quad (2) \hspace{5pt} C \cap \bar{\mathscr{D}}_{M_0^{\perp}} \neq \emptyset \quad (3) \hspace{5pt} \operatorname{Sing}C \cap \bar{\mathscr{D}}_{M_0^{\perp}} =\emptyset \\
			&(4) \hspace{5pt} \text{$C$ intersects $\bar{\mathscr{D}}_{M_0^{\perp}}$ transversally}.
		\end{align*}
		Let $\varphi : \tilde{C} \to C$ be the normalization.
		Note that $\varphi^* \partial \log u$ is a meromorphic 1-form on $\tilde{C}$ and that its residue at $p \in \varphi^{-1}(C \cap \bar{\mathscr{D}}_{M_0^{\perp}})$ is equal to $a$ by (\ref{al-2-3-3}) and $(4)$.
		By the residue theorem, we have 
		$$
			\sum_{p \in \varphi^{-1}(C \cap \bar{\mathscr{D}}_{M_0^{\perp}})} a =0.
		$$
		By $(2)$, we obtain $a=0$, and $\log u$ extends to a pluriharmonic function on $\mathcal{M}_{M_0}$.
		Since $r(M_0) \leqq 17$, we have $ \operatorname{codim} \mathcal{M}_{M_0}^* \setminus \mathcal{M}_{M_0} \geqq 2$.
		Therefore $\log u$ extends to a pluriharmonic function on $\mathcal{M}_{M_0}^*$ by \cite[Satz 4]{MR0081960}.
		Since $\mathcal{M}_{M_0}^*$ is compact, $\log u$ is constant by the maximal principle.
		This completes the proof. 
	\end{proof}

	 By Fu-Zhang \cite[Theorem A]{MR4549963}, the BCOV invariant of Calabi-Yau manifolds is a birational invariant.
	 After the result of Fu-Zhang \cite[Theorem A]{MR4549963} and Theorem \ref{t-2-3-4}, we expect that the invariant $\tau_{M, \mathcal{K}}$ is also a birational invariant.
	 
	 \begin{conj}\label{c4-3-1}
	 	Let $M$ be an admissible sublattice of $L_2$ and let $\mathcal{K}_1, \mathcal{K}_2 \in \ktm$ be K\"ahler-type chambers.
		Let $(X_1, \iota_1)$ and $(X_2, \iota_2)$ be manifolds of $K3^{[2]}$-type with involution of type $(M, \mathcal{K}_1)$ and $(M, \mathcal{K}_2)$, respectively.
		If there exists a $\G$-equivariant birational map between $(X_1, \iota_1)$ and $(X_2, \iota_2)$, then 
		$
			\tau_{M, \mathcal{K}_1}(X_1, \iota_1) =\tau_{M, \mathcal{K}_2}(X_2, \iota_2).
		$
		%Let $(X_1, \iota_1)$ and $(X_2, \iota_2)$ be manifolds of $K3^{[2]}$-type with involution of type $(M, \mathcal{K})$.
		%If there exists a $\G$-equivariant birational map between $(X_1, \iota_1)$ and $(X_2, \iota_2)$, then 
		%$
		%	\tau_{M, \mathcal{K}}(X_1, \iota_1) =\tau_{M, \mathcal{K}}(X_2, \iota_2).
		%$
	 \end{conj}

\section{Some examples}		
	 
	 By Theorems \ref{t2-1-4-1}, \ref{t-2-3-4} and \cite[Theorem 0.1]{MR4177283}, the BCOV invariant of the Calabi-Yau 4-folds of Camere-Garbagnati-Mongardi can be expressed as the Petersson norm of an explicit automorphic form obtained as the tensor product of a Borcherds product and a Siegel modular form.
	 In this section, we prove that the BCOV invariant of the Calabi-Yau 4-fold $Z$ of Camere-Garbagnati-Mongardi is expressed as the Petersson norm of a Borcherds lift if every small deformation of $Z$ remains to be a Calabi-Yau 4-fold of Camere-Garbagnati-Mongardi.

\subsection{The BCOV invariant of the universal covering of the Hilbert scheme of 2-points on an Enriques surface}
	
	\begin{exa}\label{e-2-3-5}
		Consider an Enriques surface $S$.
		Its universal covering $\widetilde{S}$ is a K3 surface
		and the covering involution $\sigma : \widetilde{S} \to \widetilde{S}$ is antisymplectic.
		By \cite[Theorem 4.2.2]{MR633160}, the $\sigma$-invariant sublattice $H^2( \widetilde{S}, \mathbb{Z} )^{\sigma}$ of $H^2( \widetilde{S}, \mathbb{Z} )$ is isomorphic to $M_0 := U(2) \oplus E_8(2)$,
		where $U$ is the even unimodular hyperbolic lattice of signature $(1,1)$
		and $E_8$ is the $E_8$-lattice, i.e., negative-definite even unimodular lattice of rank $8$.
		Here, for a lattice $L =(\mathbb{Z}^r, (\cdot, \cdot)_{L})$, we define $L(k) =(\mathbb{Z}^r, k(\cdot, \cdot)_{L})$. 
		Therefore $(\widetilde{S}, \sigma)$ is a 2-elementary K3 surface of type $M_0$.
		By \cite[Theorem 3.1]{MR2863907} and \cite[6.1]{MR3928256}, $Z_{\widetilde{S}}=\widetilde{S^{[2]}}$, where $\widetilde{S^{[2]}}$ is the universal covering of the Hilbert scheme of 2-points on the Enriques surface $S$.
		
		By \cite[Theorem 5.1 and A]{MR3928256}, we have
		$
			\dim H^1(Z_{\widetilde{S}}, T_{Z_{\widetilde{S}}}) = \dim H^3(Z_{\widetilde{S}}, \Omega^1_{Z_{\widetilde{S}}}) =10.
		$
		Therefore the dimension of the Kuranishi space of $Z_{\widetilde{S}}$ is equal to that of $S$
		and every small deformation of $Z_{\widetilde{S}}$ is induced from a small deformation of $S$.
	\end{exa}
	
	Let us recall the construction of the Borcherds $\Phi$-function \cite{MR1396773}.
	Set $\Lambda = U \oplus E_8(2)$.
	Then $M_0^{\perp} = U(2) \oplus \Lambda$,
	where the orthogonal complement is considered in $L_{K3}$.
	An element of $M_0^{\perp}$ is denoted by $(m, n, \lambda)$,
	where $m, n \in \mathbb{Z}$ and $\lambda \in \Lambda$.
	
	Set
	$
		\mathscr{C}_{\Lambda} = \{ s \in \Lambda_{\mathbb{R}}; x^2>0 \}.
	$
	The tube domain $\Lambda_{\mathbb{R}} + i \mathscr{C}_{\Lambda}$ can be identified with the period domain $\Omega_{M_0^{\perp}}$ via the isomorphism
	$
		\Lambda_{\mathbb{R}} + i \mathscr{C}_{\Lambda} \to \Omega_{M_0^{\perp}} \quad y \mapsto \left[ \left( \frac{1}{2}, -\frac{y^2}{2}, y \right) \right].
	$
	Set $\rho =(0, 1, 0), \rho' =(1, 0, 0) \in \Lambda$.
	The connected component of $\mathscr{C}_{\Lambda}$ whose closure contains $\rho$ is denoted by $\mathscr{C}_{\Lambda}^+$.
	
	Let $\mathbb{H}$ be the upper half plane.
	Set $q = e^{2 \pi i \tau}$ for $\tau \in \mathbb{H}$.
	The Dedekind $\eta$-function is defined by
	$$
		\eta(\tau) = q^{\frac{1}{24}} \prod_{n > 0} (1-q^n).
	$$
	Let $\{ c(n) \}_{n \geqq -1}$ be the series defined by the generating function
	$$
		\sum_{n \geqq -1} c(n) q^n = \eta(\tau)^{8} \eta(2 \tau)^{-8} \eta(4 \tau)^{8}. %= \frac{ \eta(2 \tau)^{8}}{\eta(\tau)^{8} \eta(4 \tau)^{8}}.
	$$
	Set
	$
		\Pi^+ = \mathbb{Z}_{>0} \rho \cup \{ r \in \Lambda; (\rho, r)>0 \}.
	$
	We define the infinite product $\Phi(y)$ by
	$$
		\Phi(y) = e^{2 \pi i (\rho, y)} \prod_{r \in \Pi^+} \left( 1-e^{2 \pi i (r,y)} \right)^{ (-1)^{(r, \rho-\rho' )} c \left( \frac{r^2}{2} \right) } .
	$$
	By \cite[Theorem 3.2]{MR1396773} and \cite[Theorem 13.3 and Example 13.7]{MR1625724}, $\Phi(y)$ converges absolutely for $y \in \Lambda_{\mathbb{R}} + i \mathscr{C}_{\Lambda}^+$ with $(\operatorname{Im} y)^2 \gg 0$
	and extends to an automorphic form on $\Omega_{M_0^{\perp}}^+ (\cong \Lambda_{\mathbb{R}} + i \mathscr{C}_{\Lambda}^+ )$ for $O^+(M_0^{\perp})$ of weight $4$ with zero divisor $\mathscr{D}_{M_0^{\perp}} \cap \Omega_{M_0^{\perp}}^+$.
	
	Fix a vector $l_{\Lambda}$ with $l_{\Lambda}^2 \geqq 0$.
	Set 
	$
		K_{\Lambda}([z]) = (z, \bar{z})/| (z, l_{\Lambda}) |^2.
	$
	For an automorphic form $f$ on $\Omega^+_{M_0^{\perp}}$ for $O^+(M_0^{\perp})$ of weight $w$, we define the Petersson norm $\| f \|$ by
	$$
		\| f ([z]) \|^2 = K_{\Lambda}([z])^w | f ([z]) |^2.
	$$
	Then $\| f \|$ is a $C^{\infty}$ function on $\mathcal{M}_{M_0}^{\circ} = \Omega^+_{M_0^{\perp}} / O^+(M_0^{\perp})$.
	
	\begin{thm}\label{t2-2-3-2}
		There exists a constant $C>0$ such that for any Enriques surface $S$
		\begin{align*}
			\tau_{BCOV}( \widetilde{S^{[2]}} ) = C \| \Phi ([S]) \|,
		\end{align*}
		where $\| \Phi ([S]) \|$ is the Petersson norm of $\Phi$ evaluated at the period of $S$.
	\end{thm}
	
	\begin{proof}
		Note that $\mathcal{M}_{M_0}^{\circ}$ is a coarse moduli space of Enriques surfaces.
		By \cite[(1.10)]{MR771979}, $\bar{\mathscr{D}}_{M_0^{\perp}}$ is irreducible.
		%Moreover, we have $\rk(M_0) =10 \leqq 17$.
		By Theorems \ref{t-2-3-4} and \ref{t2-1-4-1}, there exist constants $C_1$, $C_2$ such that
		\begin{align}\label{al-2-3-6}
			\tau_{BCOV}(\widetilde{S^{[2]}}) =C_1 \tau_{\widetilde{M}_0, \mathcal{K}} (X_{\widetilde{S}}, \iota_{\widetilde{S}}), \quad  \tau_{\widetilde{M}_0, \mathcal{K}} (X_{\widetilde{S}}, \iota_{\widetilde{S}}) = C_2  \tau_{M_0} (\widetilde{S}, \sigma)^{-2}
		\end{align} 
		for any Enriques surface $S$.
		By \cite[Theorem 8.3]{MR2047658}, there is a constant $C_3$ such that
		\begin{align}\label{al-2-3-7}
			\tau_{M_0} (\widetilde{S}, \sigma) = C_3 \| \Phi ([S]) \|^{-\frac{1}{2}}
		\end{align}
		for any Enriques surface $S$.
		Combining (\ref{al-2-3-6}) and (\ref{al-2-3-7}), we obtain the desired result.
	\end{proof}

\subsection{The case where $Y^{\sigma}$ is a union of rational curves}

	  \begin{exa}\label{e-2-3-8}
		Let $k \in \{1, \dots, 6\}$ %and set
		%$$
		%	\Lambda_k = I_2 \oplus I_{10-k}(-1),
		%$$
		%where $I_N$ is a positive-definite odd unimodular lattice of rank $N$.
		%Namely, $\Lambda_k$ is an odd unimodular lattice of signature $(2,10-k)$.
		and let $\Lambda_k$ be an odd unimodular lattice of signature $(2,10-k)$.
	 	Set 
		$
			M_0=\Lambda_k(2)^{\perp}.
		$
		Here the orthogonal complement is considered in $L_{K3}$.
		%We define
		%$$
		%	(r, l ) := (\rk (M_0), \dim_{\mathbb{Z}_2} ((M_0)^{\vee} / M_0) ) = (10+k, 12-k)
		%$$
		%and 
		%$$
		%	g := 11-\frac{r+l}{2} = 0.
		%$$
		By \cite[Theorem 4.2.2.]{MR633160}, the fixed locus $Y^{\sigma}$ of a 2-elementary K3 surface $(Y, \sigma)$ of type $M_0=\Lambda_k(2)^{\perp}$ is the union of $k$-rational curves. 
		In this case, $\bar{\mathscr{D}}_{M_0^{\perp}}$ is also irreducible by \cite[11.3]{MR3039773} and \cite[FIGURE 1]{MR2424924}.
		
		By \cite[Theorem 5.1 and A]{MR3928256}, we have
		$
			\dim H^1(Z_{Y}, T_{Z_{Y}}) = \dim H^3(Z_{Y}, \Omega^1_{Z_{Y}}) =10-k.
		$
		Since
		$
			\dim \Omega_{M_0^{\perp}} = 10-k,
		$
		the dimension of the Kuranishi space of $Z_{Y}$ is equal to that of $(Y, \sigma)$.
		Hence every small deformation of $Z_{Y}$ is induced from a deformation of the 2-elementary K3 surface $(Y, \sigma)$ of type $M_0$.
	\end{exa}

	Let $\Phi_k$ be the reflexive modular form on $\Omega_{M_0^{\perp}}$ characterizing the discriminant locus $\mathscr{D}_{M_0^{\perp}}$.
	As before, $\Phi_k$ is a certain Borcherds product and has an infinite product expansion.
	For the explicit construction, we refer \cite[Theorem 4.2]{MR2674883}, \cite[Theorem 5.4]{MR3859399}.
	
	\begin{thm}\label{t2-2-3-4}
		There exists a constant $C_k$ depending only on $1 \leqq k \leqq 6$ such that
		\begin{align*}
			\tau_{BCOV}(Z_Y) = C_k \| \Phi_k ([Y]) \|^{k+1}
		\end{align*}
		for any 2-elementary K3 surface $(Y, \sigma)$ of type $M_0$, where $\| \Phi_k ([Y]) \|$ is the Petersson norm of $\Phi_k$ evaluated at the period of $(Y, \sigma)$.
	\end{thm}
	
	\begin{proof}
		By Theorem \ref{t-2-3-4} and \ref{t2-1-4-1}, there exist constants $C_{k,1}$, $C_{k,2}$ depending only on $k$ such that for any 2-elementary K3 surface $(Y, \sigma)$ of type $M_0$
		\begin{align}\label{al-2-3-9}
			\tau_{BCOV}(Z_Y) =C_{k,1} \tau_{\widetilde{M}_0, \mathcal{K}} (X_Y, \iota_Y), \quad  \tau_{\widetilde{M}_0, \mathcal{K}} (X_Y, \iota_Y) = C_{k,2}  \tau_{M_0} (Y, \sigma)^{-2(k+1)}.
		\end{align} 
		By \cite[Theorem 4.2.(1)]{MR2674883} and \cite[Theorem 0.1]{MR3039773}, there exists a constant $C_{k,3}$ with
		\begin{align}\label{al-2-3-10}
			\tau_{M_0} (Y, \sigma) = C_{k,3} \| \Phi_k ([Y]) \|^{-\frac{1}{2}}
		\end{align}
		for any 2-elementary K3 surface $(Y, \sigma)$ of type $M_0$.
		Combining (\ref{al-2-3-9}) and (\ref{al-2-3-10}), we obtain the desired result.
	\end{proof}

	 According to the genus one mirror symmetry conjecture (\cite[Conjectures 1 and 2]{MR4475251}), for a Calabi-Yau manifold, the BCOV invariant of its mirror family is equivalent to a certain generating series of the genus one Gromov-Witten invariants of the Calabi-Yau manifold.
	 For Calabi-Yau hypersurfaces in the projective spaces, this conjecture is verified by Zinger \cite{MR2403808} \cite{MR2505298}, Fang-Lu-Yoshikawa \cite{MR2454893} in dimension $3$, and by Zinger \cite{MR2403808} \cite{MR2505298}, Eriksson-Freixas i Montplet-Mourougane \cite{MR4475251} in arbitrary dimension.
	 For Enriques surfaces, an analogue of the genus one mirror symmetry conjecture holds true by Oberdieck \cite{oberdieck2023non} and Yoshikawa \cite{MR2047658}.
	 For a smooth complete intersection of two cubic hypersurfaces of $\mathbb{P}^5$, the genus one mirror symmetry conjecture is verified by Popa \cite{MR3003261} and Eriksson-Pochekai \cite{eriksson-pochekai}.
	 
	At least to the author, it is unclear what is the mirror of the Calabi-Yau 4-folds in Examples \ref{e-2-3-5} and \ref{e-2-3-8}.
	According to Camere-Garbagnati-Mongardi \cite[Corollary 5.8]{MR3928256}, even if $(\check{Y}, \check{\sigma})$ is a lattice theoretic mirror of a 2-elementary K3 surface $(Y, \sigma)$ of type $M_0$, the Hodge diamond of $Z_{\check{Y}}$ is not mirror to the one of $Z_{Y}$.
	Namely, $h^{1,1}(Z_Y) \neq h^{3,1}(Z_{\check{Y}})$.
	
	% Camere-Garbagnati-Mongardi \cite{MR3928256} pointed out that the mirror symmetry at the level of a 2-elementary K3 surface $(Y, \sigma)$ does not induce a Calabi-Yau 4-fold which is mirror symmetric to $Z_Y$.
	% They also showed that the lattice theoretic mirror symmetry at the level of a hyperk\"ahler 4-fold $Y^{[2]}$ does not induce a Calabi-Yau 4-fold that is mirror symmetric to $Z_Y$.
	 
	 Let $Z$ be the Calabi-Yau 4-fold in Examples \ref{e-2-3-5} and \ref{e-2-3-8}.
	 By Theorems \ref{t2-2-3-2} and \ref{t2-2-3-4}, we expect that if there exists a mirror Calabi-Yau 4-fold of $Z$, 
	 then its genus one Gromov-Witten invariants are encoded by a power of $\Phi$ or $\Phi_k$ in a certain way.

\bibliography{Reference}
\bibliographystyle{amsplain}

\end{document}